\documentclass{article}
\usepackage[margin=1in]{geometry}
\usepackage{graphicx}
\usepackage{amsmath}
\usepackage{amssymb}
\usepackage{amsthm}
\usepackage{ytableau}
\usepackage{tikz}
\usepackage{float}
\usepackage{cite}
\usepackage[colorlinks=true,linkcolor=blue!60!black,citecolor=green!40!black,urlcolor=blue!60!black]{hyperref}

\theoremstyle{definition}
\newtheorem{definition}{Definition}[section]
\newtheorem{example}[definition]{Example}
\newtheorem{theorem}[definition]{Theorem}
\newtheorem{remark}[definition]{Remark}
\newtheorem{corollary}[definition]{Corollary}
\newtheorem{proposition}[definition]{Proposition}
\newtheorem{lemma}[definition]{Lemma}

\ytableausetup{boxsize=1.2em}

\title{From stylic monoid to Catalan monoid}
\author{Itamar Stein \\ Mathematics Unit, Shamoon College of Engineering \\
\href{mailto:steinita@gmail.com}{\texttt{steinita@gmail.com}}}
\date{}

\begin{document}

\maketitle

\begin{abstract}
The stylic monoid $\mathrm{Styl}_n$, introduced by Abram and
Reutenauer, is the quotient of the plactic monoid by the relations
$x^2=x$, and its elements are represented by $N$-tableaux. Volkov
showed that the Catalan monoid $\mathrm{Cat}_n$ of order-preserving,
order-decreasing self-maps of $\{0,1,\ldots,n\}$ is a quotient of
$\mathrm{Styl}_n$. However, the quotient map is defined on generators,
and it is not apparent how to see, from an $N$-tableau, the map in
$\mathrm{Cat}_n$ it corresponds to. In this paper we give a simple
visual way to read off this map, and some of its main properties,
from the $N$-tableau.

The new ingredient is that we do not insist on drawing an $N$-tableau
as a classical Young tableau: we allow the entries of each row to be
shifted relative to the row below, as long as each entry stays above
a smaller one. We call this a positioning, and prove that the column
word read from any positioning is plactically equivalent to the usual
column word; so every positioning can be used to compute the quotient
map. We work with the tight positioning, in which each entry is pushed
as far right as possible, and define the full core of an $N$-tableau:
the part of each column that climbs by consecutive values from the
bottom row. We call an $N$-tableau full if it equals its full core.
We prove that passing to the full core does not change the
image in $\mathrm{Cat}_n$, that full $N$-tableaux are in bijection with
$\mathrm{Cat}_n$, and we show how to read the corresponding map
directly off a full $N$-tableau.


\end{abstract}

\section{Introduction}
\label{sec:intro}

The stylic monoid $\mathrm{Styl}_n$ was introduced by Abram and
Reutenauer \cite{AR}. The free monoid $\mathbf{n}^*$ over
$\mathbf{n}=\{1,\ldots,n\}$ acts on the left on columns by Schensted
left insertion, and $\mathrm{Styl}_n$ is the transition monoid of this
action. Abram and Reutenauer showed that $\mathrm{Styl}_n$ is a
quotient of the plactic monoid: it is generated by $\mathbf{n}$ subject
to the Knuth relations ($yxz=yzx$ for $x<y\le z$ and $xzy=zxy$ for
$x\le y<z$) together with the idempotent relations $x^2=x$. This
presentation is the description of $\mathrm{Styl}_n$ that we use
throughout the paper. The elements of $\mathrm{Styl}_n$ are represented
visually by $N$-tableaux: semistandard tableaux with strictly
increasing rows, each row contained in the row below it. Among other
results, Abram and Reutenauer showed that $|\mathrm{Styl}_n|$ is the
number of partitions of a set with $n+1$ elements, that
$\mathrm{Styl}_n$ is $\mathcal{J}$-trivial, and that it is the
syntactic monoid of the length of the longest strictly decreasing
subword.

Being a natural finite quotient of the plactic monoid, the stylic
monoid has attracted the attention of several authors since its
introduction. Topics studied include its identities
\cite{Volkov,AirdRibeiro}, its representation theory
\cite{ARS}, similar quotients of other plactic-like monoids
\cite{AHMNT}, and coherent presentations \cite{Hage}. It also inspired the
grammic monoid, defined analogously by an action on rows
\cite{Choffrut}.

The Catalan monoid $\mathrm{Cat}_n$ is the monoid of all
order-preserving, order-decreasing self-maps of the chain
$\{0,1,\ldots,n\}$. It is a classical object in its own right: its cardinality is the
$(n+1)$-st Catalan number, and it has been studied in semigroup theory as a monoid of
transformations of a chain \cite{Higgins93}, \cite[Chapter~14]{GM},
and in algebraic combinatorics as the monoid of non-decreasing parking
functions \cite{HivertThiery}. Like $\mathrm{Styl}_n$, it is
$\mathcal{J}$-trivial.

The starting point of this paper is an observation made by Volkov in
\cite{Volkov}: $\mathrm{Cat}_n$ is a homomorphic image of
$\mathrm{Styl}_n$. His argument works with
presentations. Solomon \cite[Section~9]{Solomon} showed that $\mathrm{Cat}_n$ is
generated by $a_1,\ldots,a_n$ subject to the relations $a_i^2=a_i$,
$a_ia_k=a_ka_i$ for $|i-k|\ge2$, and
$a_ia_{i+1}a_i=a_{i+1}a_ia_{i+1}=a_{i+1}a_i$; Volkov checks that the
relations of a presentation of $\mathrm{Styl}_n$ from \cite{AR} hold
in $\mathrm{Cat}_n$, so
that sending each generator of $\mathrm{Styl}_n$ to the corresponding
generator of $\mathrm{Cat}_n$ extends to a surjective homomorphism
$\Phi:\mathrm{Styl}_n\twoheadrightarrow\mathrm{Cat}_n$.
Since $\Phi$ is defined on generators, this argument does not show which map in $\mathrm{Cat}_n$
corresponds to a given $N$-tableau.

Our goal is to describe $\Phi$ visually: to look at an $N$-tableau $T$
and read off directly the map $\Phi(T)\in\mathrm{Cat}_n$. Of course,
$\Phi(T)$ can always be computed by composing the generators of
$\mathrm{Cat}_n$ along a word representing $T$, such as its row word;
the point is to see the answer in the picture of $T$ instead.

Section~\ref{sec:preliminaries} recalls the necessary background on
the plactic and stylic monoids. The main new ingredient is some freedom in how an $N$-tableau is drawn.
Classically, the rows of a tableau are left-justified. We allow each
row to be shifted relative to the row below it, entry by entry, as long
as every entry stays above a smaller one; we call such a placement a
\emph{positioning} (Section~\ref{sec:alignment}). A positioning changes
the columns, and hence the column word, but we show that it does not
change its plactic class (Theorem~\ref{thm:col-invariant}, proved in
Section~\ref{sec:col-plactic} for arbitrary row arrays). So any
positioning of an $N$-tableau $T$ can be used to compute $\Phi(T)$.

We work with the \emph{tight positioning}, in which each entry sits
above the largest smaller entry of the row below
(Section~\ref{sec:tight}). Read upward from the bottom row, a column
of the tight positioning may climb by consecutive values all the way
up, or it may at some row jump by more than one; we call such a jump a
\emph{gap}. Keeping, in every column, only the cells below its first
gap gives the \emph{full core} $T^f$ of $T$, which is again an $N$-tableau
(Section~\ref{sec:reduction}). The main result of
Section~\ref{sec:catalan-image} is that $\Phi(T)=\Phi(T^f)$
(Theorem~\ref{thm:phi-full-core}): the cells above the gaps do not
affect the image in $\mathrm{Cat}_n$. We call $T$ \emph{full} if
$T=T^f$.

In Section~\ref{sec:full-image} we describe $\Phi(T)$ for a full
$N$-tableau $T$. A letter that does
not occur in $T$ is fixed, and a letter $y$ that does occur is sent to
one less than the bottom entry of the leftmost column containing $y$
(Theorem~\ref{thm:phi-full}). Adding a bottom row $0,1,\ldots,n$ makes
this rule uniform (Corollary~\ref{cor:phi-full-augmented}), and from it
we show that $\Phi$ is injective on full $N$-tableaux. Hence full
$N$-tableaux are in bijection with $\mathrm{Cat}_n$, and $T^f$ is the
unique full $N$-tableau with the same image as $T$
(Corollary~\ref{cor:full-bijection}). We then read off the kernel and
image of $\Phi(T)$ from the tops and bases of the columns
(Proposition~\ref{prop:image-kernel}), and, redrawing the tableau as a
wall of pyramids, show that $\Phi(T)$ is idempotent if and only if the
pyramids stand on pairwise disjoint bases
(Proposition~\ref{prop:idempotents}). The paper ends with a complete
example (Section~\ref{sec:complete-example}), which follows a single
$N$-tableau through all the steps.

\section{Preliminaries}
\label{sec:preliminaries}

Fix $n\ge1$; our alphabet throughout is $\mathbf{n}:=\{1,\ldots,n\}$,
with its usual order.

\subsection{The plactic monoid}
\label{sec:prelim-plactic}

The \emph{Knuth relations} on $\mathbf{n}^*$ are
\[
\text{(K1)}\quad yxz \equiv_{\mathrm{plac}} yzx \quad (x<y\le z), \qquad\qquad
\text{(K2)}\quad xzy \equiv_{\mathrm{plac}} zxy \quad (x\le y<z),
\]
for letters $x,y,z\in\mathbf{n}$. Two words
$u,v\in\mathbf{n}^*$ are \emph{plactically
equivalent}, written $u\equiv_{\mathrm{plac}} v$, if $v$ can be
obtained from $u$ by a finite sequence of such relations, each applied
to three consecutive letters. This is a congruence on $\mathbf{n}^*$, and the
\emph{plactic monoid} is $\mathrm{Plac}_n := \mathbf{n}^*/{\equiv_{\mathrm{plac}}}$.

Recall that a Young diagram is the diagram of a partition: finitely
many rows of boxes with weakly decreasing lengths. We adopt the
French convention throughout, drawing rows with lengths weakly
decreasing from bottom to top. A \emph{semistandard Young
tableau (SSYT)} over $\mathbf{n}$ is then a filling of such a diagram by
elements of $\mathbf{n}$ such that entries weakly increase along each row (left
to right) and strictly increase up each column (bottom to top). There is
a bijection between $\mathrm{Plac}_n$ and the set of SSYT over $\mathbf{n}$,
i.e.\ every plactic class has a unique SSYT representative. In one
direction: given a word $w\in\mathbf{n}^*$, the SSYT representing its plactic
class is $P(w)$, obtained by the classical (Schensted) row-insertion
algorithm -- inserting the letters of $w$ one at a time into the empty
tableau, where each new letter either extends a row or \emph{bumps}
the smallest strictly larger entry of that row up into the row above.
We will not need further details of this classical construction. In
the other direction: given an SSYT $T$, a word representing its
plactic class is recovered by \emph{reading} $T$ -- either its
\emph{row word}, obtained by reading each row left to right and
concatenating the rows from top to bottom, or its \emph{column word},
obtained by reading each column top to bottom (necessarily a strictly
decreasing sequence of entries) and concatenating the columns from
left to right. The row word and column word of an SSYT are in general
different words, but -- since both represent the same plactic class --
they are always plactically equivalent.
More on the plactic monoid and Knuth equivalence can be found in \cite[Chapter~5]{Lothaire}.
\begin{remark}
Later in this paper we redefine SSYT
(Definition~\ref{def:ssyt-as-aligned}), in a form that is equivalent
to the classical one above but better suited to the framework of this
paper.
\end{remark}

\subsection{The stylic monoid}

The \emph{stylic monoid} $\mathrm{Styl}_n$, introduced by Abram and
Reutenauer \cite{AR}, is the further quotient of $\mathrm{Plac}_n$
obtained by imposing the idempotent relations
$x^2=x$ for every $x\in\mathbf{n}$; equivalently, $\mathrm{Styl}_n :=
\mathbf{n}^*/{\equiv_{\mathrm{styl}}}$, where $\equiv_{\mathrm{styl}}$ is the
congruence on $\mathbf{n}^*$ generated by both the Knuth relations and these
idempotent relations. Since $\equiv_{\mathrm{plac}}$ refines
$\equiv_{\mathrm{styl}}$,
there is a well-defined surjective monoid homomorphism
$\mathrm{Plac}_n\twoheadrightarrow\mathrm{Styl}_n$.

An \emph{$N$-tableau} over $\mathbf{n}$ is an SSYT over $\mathbf{n}$ with $m$ rows
$R_1,\ldots,R_m$ satisfying two extra conditions. First, each row is
strictly increasing. Such a row may be identified with the set of its
entries, with no loss of information, since the row lists this set in
increasing order. Under this identification, the second condition is
$R_i\subseteq R_{i-1}$ for every $1<i\le m$. Conversely, it is
noted in \cite[Section~6.1]{AR} that any finite sequence of rows
$R_1,\ldots,R_m$, each a finite strictly increasing sequence of
elements of $\mathbf{n}$, such that $R_i\subseteq R_{i-1}$ for every $1<i\le m$
and $\min R_1<\min R_2<\cdots<\min R_m$, yields a unique
$N$-tableau.

The set of $N$-tableaux over $\mathbf{n}$ is in bijection with
$\mathrm{Styl}_n$ \cite[Theorem~7.1]{AR}. In one direction, given a
word $w\in\mathbf{n}^*$ representing an element of $\mathrm{Styl}_n$, the
corresponding $N$-tableau is $N(w)$, computed by an insertion
algorithm whose details we will not need. Two words represent the
same element of $\mathrm{Styl}_n$ if and only if they produce the
same $N$-tableau under $N(\cdot)$.

In the other direction, given an $N$-tableau $T$, a word representing
the corresponding element of $\mathrm{Styl}_n$ is recovered exactly
as in the plactic case: either by reading its row word $\mathrm{row}(T)$
(read top row to bottom row, each row left to right), or by reading
its column word (read left column to right column, each column top to
bottom). Both recover $T$:
\[
N(\mathrm{row}(T)) = N(\mathrm{col}(T)) = T.
\]
The row-word equality is \cite[Lemma~7.5]{AR}; the column-word
equality then follows from
$\mathrm{row}(T)\equiv_{\mathrm{plac}} \mathrm{col}(T)$
(Section~\ref{sec:prelim-plactic} above) together with
$\equiv_{\mathrm{plac}}\ \subseteq\ \equiv_{\mathrm{styl}}$.

\section{Alignment of row arrays}
\label{sec:alignment}

\subsection{Positioning}

\begin{definition}[Row array]
\label{def:row-array}
A \emph{row} over a totally ordered alphabet $\mathbf{n}$ is a finite weakly
increasing word $R = x_1 \le x_2 \le \dots \le x_\ell$ with $x_i \in\mathbf{n}$
and $\ell\ge0$; we write $|R| := \ell$ for the \emph{length} of $R$, i.e.\
its number of entries. The case $\ell=0$, the \emph{empty row}
$R=\varnothing$, is allowed: the rows we actually care about are
essentially always nonempty, but permitting the empty row is convenient for some constructions later in the
paper (e.g. the full core of Definition~\ref{def:full-core}). A \emph{row array}
is a finite sequence $(R_1,\dots,R_m)$ of rows over $\mathbf{n}$. For each $i$,
write the entries of $R_i$ as
\[
R_{i,1} \le R_{i,2} \le \dots \le R_{i,|R_i|}.
\]
No relation between $R_i$ and $R_{i+1}$ (containment, size, etc.) is
assumed.
\end{definition}

We picture a row array as stacked bottom to top, with $R_1$ the
\emph{bottom} row and $R_m$ the \emph{top} row: the row index
increases as we go up.

\begin{definition}[Positioning]
\label{def:positioning}
Let $(R_1,\dots,R_m)$ be a row array. A \emph{positioning} is a family
of functions
\[
\rho_i : \{1,\dots,|R_i|\} \to \mathbb{N}, \qquad 2 \le i \le m,
\qquad 1 \le \rho_i(j) \le |R_{i-1}|,
\]
such that the following two conditions hold for every $i>1$:
\begin{enumerate}
    \item[(P1)] \textbf{(Value condition)} $R_{i-1,\, \rho_i(j)} < R_{i,j}$ for every $j$;
    \item[(P2)] \textbf{(Strict monotonicity)} $j < j' \implies \rho_i(j) < \rho_i(j')$.
\end{enumerate}
\end{definition}

\begin{corollary}
\label{cor:rho-dominates-id}
For every positioning $\rho$ on a row array $(R_1,\dots,R_m)$, every
$i>1$, and every $1\le j\le|R_i|$,
\[
\rho_i(j)\ge j.
\]
\end{corollary}

\begin{proof}
By (P2), $\rho_i$ is strictly increasing on $\{1,\dots,|R_i|\}$;
induction on $j$ gives $\rho_i(j)>\rho_i(j-1)\ge j-1$, so $\rho_i(j)\ge j$.
\end{proof}

\begin{definition}[Aligned row array]
\label{def:aligned-row-array}
An \emph{aligned row array} is a pair $(T,\rho)$ where $T=(R_1,\dots,R_m)$
is a row array and $\rho$ is a positioning on $T$.
\end{definition}

We think of an aligned row array as arranging its rows stacked one
above another, so that $\rho_i(j)$ records \emph{above which box of
row $i-1$} the box $(i,j)$ sits, and (P1) requires the entry in that
box below to be strictly smaller.

\begin{example}
\label{ex:aligned-picture}
Let $R_1=(1,2,3,4,7)$, $R_2=(4,7,8)$, $R_3=(9)$, with positioning
\[
\rho_2(1)=2,\ \rho_2(2)=4,\ \rho_2(3)=5, \qquad\qquad
\rho_3(1)=2.
\]
So $R_{2,1}=4$ sits above $R_{1,2}=2$, $R_{2,2}=7$ above $R_{1,4}=4$,
and $R_{2,3}=8$ above $R_{1,5}=7$; while $R_{3,1}=9$ sits above
$R_{2,2}=7$.
Figure~\ref{fig:aligned-picture} draws the three
rows stacked directly on top of each other, as in the row-indexing
convention.
\begin{figure}[H]
\centering
\begin{tikzpicture}[scale=0.9, every node/.style={font=\small}]
\foreach \i/\v in {1/1,2/2,3/3,4/4,5/7} {
  \draw (\i-1,0) rectangle (\i,1);
  \node at (\i-0.5,0.5) {$\v$};
}
\foreach \i/\v in {2/4,4/7,5/8} {
  \draw (\i-1,1) rectangle (\i,2);
  \node at (\i-0.5,1.5) {$\v$};
}
\foreach \i/\v in {4/9} {
  \draw (\i-1,2) rectangle (\i,3);
  \node at (\i-0.5,2.5) {$\v$};
}
\node[anchor=east] at (-0.3,0.5) {$R_1$};
\node[anchor=east] at (-0.3,1.5) {$R_2$};
\node[anchor=east] at (-0.3,2.5) {$R_3$};
\end{tikzpicture}
\caption{The aligned row array of Example~\ref{ex:aligned-picture}.}
\label{fig:aligned-picture}
\end{figure}
\end{example}

\begin{lemma}
\label{lem:sizes-decrease}
If a positioning exists on a row array $(R_1,\dots,R_m)$, then
$|R_1|\ge|R_2|\ge\cdots\ge|R_m|$.
\end{lemma}

\begin{proof}
Fix $i>1$ and let $\rho$ be a positioning on $(R_1,\dots,R_m)$. By
(P2), $\rho_i$ is a strictly increasing map
$\{1,\dots,|R_i|\}\to\{1,\dots,|R_{i-1}|\}$, hence injective, so
$|R_i|\le|R_{i-1}|$.
\end{proof}

\begin{definition}[Trivial positioning]
\label{def:trivial-positioning}
For a row array $(R_1,\dots,R_m)$, define $\rho^{\mathrm{tr}}_i(j)
:= j$ for $2\le i\le m$. Condition (P2) holds automatically. When
$\rho^{\mathrm{tr}}$ also satisfies (P1), i.e.\ $R_{i-1,j} < R_{i,j}$
for every $i>1$ and $j$, we call it \emph{the trivial positioning} on
$(R_1,\dots,R_m)$. In view of Lemma~\ref{lem:sizes-decrease}, it is
clear that the trivial positioning does not always exist.
\end{definition}

\begin{example}
\label{ex:trivial-picture}
With $R_1=(1,4,6,9,10)$, $R_2=(4,5,7)$, $R_3=(5,9)$,
$\rho^{\mathrm{tr}}$ is a valid positioning on $(R_1,R_2,R_3)$.
Figure~\ref{fig:trivial-picture} draws the three rows stacked directly
on top of each other: in trivial positioning, everything is
pushed as far left as possible.
\begin{figure}[H]
\centering
\begin{tikzpicture}[scale=0.9, every node/.style={font=\small}]
\foreach \i/\v in {1/1,2/4,3/6,4/9,5/10} {
  \draw (\i-1,0) rectangle (\i,1);
  \node at (\i-0.5,0.5) {$\v$};
}
\foreach \i/\v in {1/4,2/5,3/7} {
  \draw (\i-1,1) rectangle (\i,2);
  \node at (\i-0.5,1.5) {$\v$};
}
\foreach \i/\v in {1/5,2/9} {
  \draw (\i-1,2) rectangle (\i,3);
  \node at (\i-0.5,2.5) {$\v$};
}
\node[anchor=east] at (-0.3,0.5) {$R_1$};
\node[anchor=east] at (-0.3,1.5) {$R_2$};
\node[anchor=east] at (-0.3,2.5) {$R_3$};
\end{tikzpicture}
\caption{The trivial positioning of Example~\ref{ex:trivial-picture}.}
\label{fig:trivial-picture}
\end{figure}
\end{example}

\begin{lemma}
\label{lem:positioning-implies-trivial}
If a positioning $\rho$ exists on a row array $(R_1,\dots,R_m)$, then
the trivial positioning $\rho^{\mathrm{tr}}$ also exists on
$(R_1,\dots,R_m)$.
\end{lemma}

\begin{proof}
Fix $i>1$. By Lemma~\ref{lem:sizes-decrease},
$|R_i|\le|R_{i-1}|$, so $\rho^{\mathrm{tr}}_i(j):=j$ is a
well-defined map into $\{1,\dots,|R_{i-1}|\}$, and (P2) holds trivially
for it. By Corollary~\ref{cor:rho-dominates-id}, $\rho_i(j)\ge j$
for every $j$ and since $R_{i-1}$ is weakly increasing, $R_{i-1,j} \le
R_{i-1,\rho_i(j)}$. By (P1) for $\rho$, $R_{i-1,\rho_i(j)} <
R_{i,j}$ so this proves $R_{i-1,j} < R_{i,j}$, which is (P1) for
$\rho^{\mathrm{tr}}$.
\end{proof}

We phrase the next definition in a form somewhat different from the
usual one, adapted to this framework of row arrays and positionings.

\begin{definition}[SSYT]
\label{def:ssyt-as-aligned}
A \emph{semistandard Young tableau (SSYT)} is a row array
$T=(R_1,\dots,R_m)$ with no empty rows, for which the trivial
positioning $\rho^{\mathrm{tr}}$ (Definition~\ref{def:trivial-positioning})
exists. By Lemma~\ref{lem:positioning-implies-trivial}, this is
equivalent to $T$ having no empty rows and admitting some positioning
at all.
\end{definition}

\begin{remark}
This matches the classical picture of Section~\ref{sec:preliminaries}:
rows are already weakly increasing (Definition~\ref{def:row-array}),
and admitting $\rho^{\mathrm{tr}}$ is exactly what forces a genuine
Young-diagram shape, filled consistently with it.
\end{remark}

\begin{definition}[Order on positionings]
\label{def:positioning-order}
Let $T=(R_1,\dots,R_m)$ be a row array. For positionings $\rho,\rho'$
on $T$, say $\rho \le \rho'$ if $\rho_i(j)\le\rho'_i(j)$ for every
$i>1$ and $1\le j\le|R_i|$ (the pointwise order). This is a partial
order on the set of positionings on $T$.
\end{definition}

\begin{corollary}
Whenever it exists, $\rho^{\mathrm{tr}}$ is the minimum positioning
on $T$ with respect to $\le$: $\rho^{\mathrm{tr}}\le\rho$ for every
positioning $\rho$ on $T$.
\end{corollary}

\begin{proof}
Fix $i>1$ and $1\le j\le|R_i|$. By Corollary~\ref{cor:rho-dominates-id},
$\rho_i(j)\ge j$, i.e.\ $\rho^{\mathrm{tr}}_i(j)=j\le\rho_i(j)$.
\end{proof}

\subsection{Columns in aligned row arrays}

\begin{definition}[Column]
\label{def:column}
Let $(T,\rho)$ be an aligned row array, $T=(R_1,\dots,R_m)$
(Definition~\ref{def:aligned-row-array}). Define
\[
\psi(i,j) := \bigl(i-1,\, \rho_i(j)\bigr), \qquad 2\le i\le m,\ 1\le j\le|R_i|,
\]
a map sending positions in rows $2,\dots,m$ to positions in rows
$1,\dots,m-1$, and let $\psi^k$ denote its $k$-fold iterate
($\psi^0(i,j):=(i,j)$), defined for as long as the row index stays
above $1$. Since $\rho_i$ is injective for every $i$ (P2), so
is $\psi$, and hence so is every iterate $\psi^k$.

Every position $(i,j)$ of $T$ has a well-defined forward orbit under
$\psi$ reaching row $1$ after exactly $i-1$ steps; writing
$\psi^{i-1}(i,j)=(1,r)$, define $C(i,j):=r$, the
\emph{column} of $(i,j)$ -- depending, of course, on the alignment
$\rho$. Since $\psi^{i-1}(i,j)=\psi^{i-2}(\psi(i,j))$ for $i\ge2$, a
position lies in the same column as its image under $\psi$:
\[
C(\psi(i,j))=C(i,j), \qquad 2\le i\le m.
\]
For $1\le r\le|R_1|$, write
\[
C_r := \bigl\{\, (i,j) : C(i,j)=r \,\bigr\}
\]
for the set of positions in column $r$. By injectivity of the
iterates, $C_r$ contains at most one position from each row, and
it contains exactly one from row $1$, namely $(1,r)$ itself.
Every position of $T$ lies in exactly one column, so
$C_1,\dots,C_{|R_1|}$ partition the positions of $T$.
\end{definition}

A word of caution for what follows: a position $(i,j)$ carries two
different notions of horizontal placement -- its index $j$ within its
own row, and the index $r=C(i,j)$ of the column containing it --
and the two need not agree. We will need to keep this distinction in
mind throughout.

\begin{definition}[Column as a function of row]
\label{def:column-value}
Let $C_r$ be a column. If $(i,j)\in C_r$
for some $i>1$, then $\psi(i,j)=(i-1,\rho_i(j))\in C_r$ too; so, by
downward induction, $C_r$ has exactly one member in \emph{every} row from $1$
up to its topmost row, and none above it -- in particular, $|C_r|$ is
exactly that topmost row. For $1\le i\le|C_r|$, write $C_r(i)$ for the
\emph{value} of $C_r$'s member in row $i$: if $(i,j)\in C_r$, then
$C_r(i):=R_{i,j}$.
\end{definition}

Before defining column word, we record for contrast that a row array
by itself already has a well-defined row word, with no need for a
positioning.

\begin{definition}[Row word]
The \emph{row word} of a row array $T=(R_1,\dots,R_m)$ is
\[
\mathrm{row}(T) \;:=\; R_m\,R_{m-1}\cdots R_1,
\]
the concatenation of its rows, read from the top row down to the
bottom row.
\end{definition}

\begin{definition}[Column word]
\label{def:column-word}
Let $(T,\rho)$ be an aligned row array, $T=(R_1,\dots,R_m)$. The
\emph{word of a column} $C_r$, denoted $w(C_r)$, is its
sequence of entries, read from its topmost row down to row $1$; by
(P1) this word is strictly decreasing. The \emph{column word} of
$(T,\rho)$ is
\[
\mathrm{col}(T,\rho) \;:=\; w(C_1)\,w(C_2)\cdots w(C_{|R_1|}),
\]
the concatenation of the words of the columns, in order. Unlike the
row word, the column word depends on $\rho$: a row array by itself
does not say which entries of adjacent rows belong to the same
column, so it has no column word -- only an aligned row array does.
\end{definition}

\begin{remark}
The column word $\mathrm{col}(T)$ used informally in
Section~\ref{sec:preliminaries} for the physical column reading of a
Young-diagram tableau is exactly $\mathrm{col}(T,\rho^{\mathrm{tr}})$
in the notation above: it is easy to see that
the trivial positioning places entry $(i,j)$ in column $j$, exactly
where it physically sits in the diagram.
\end{remark}

\begin{example}[Column word depends on the alignment]
Let $R_1=(1,3,3,5)$ and $R_2=(5,7)$; this row array admits several
positionings. We compare two of them:
the trivial positioning $\rho^{\mathrm{tr}}$, with
$\rho^{\mathrm{tr}}_2(1)=1$, $\rho^{\mathrm{tr}}_2(2)=2$; and
$\rho$, with $\rho_2(1)=3$, $\rho_2(2)=4$.

For $\rho^{\mathrm{tr}}$ (Figure~\ref{fig:column-word-triv}):
$C_1=\bigl((2,1),(1,1)\bigr)$, $C_2=\bigl((2,2),(1,2)\bigr)$,
$C_3=\bigl((1,3)\bigr)$, $C_4=\bigl((1,4)\bigr)$, with words $51$,
$73$, $3$, $5$, so
\[
\mathrm{col}(T,\rho^{\mathrm{tr}}) = 51\,73\,3\,5.
\]
\begin{figure}[H]
\centering
\begin{tikzpicture}[scale=0.9, every node/.style={font=\small}]
\foreach \i/\v in {1/1,2/3,3/3,4/5} {
  \draw (\i-1,0) rectangle (\i,1);
  \node at (\i-0.5,0.5) {$\v$};
}
\foreach \i/\v in {1/5,2/7} {
  \draw (\i-1,1) rectangle (\i,2);
  \node at (\i-0.5,1.5) {$\v$};
}
\node[anchor=east] at (-0.3,0.5) {$R_1$};
\node[anchor=east] at (-0.3,1.5) {$R_2$};
\end{tikzpicture}
\caption{The trivial positioning $\rho^{\mathrm{tr}}$.}
\label{fig:column-word-triv}
\end{figure}

For $\rho$ (Figure~\ref{fig:column-word-other}): $C_1=\bigl((1,1)\bigr)$,
$C_2=\bigl((1,2)\bigr)$, $C_3=\bigl((2,1),(1,3)\bigr)$,
$C_4=\bigl((2,2),(1,4)\bigr)$, with words $1$, $3$, $53$, $75$, so
\[
\mathrm{col}(T,\rho) = 1\,3\,53\,75.
\]
\begin{figure}[H]
\centering
\begin{tikzpicture}[scale=0.9, every node/.style={font=\small}]
\foreach \i/\v in {1/1,2/3,3/3,4/5} {
  \draw (\i-1,0) rectangle (\i,1);
  \node at (\i-0.5,0.5) {$\v$};
}
\foreach \i/\v in {3/5,4/7} {
  \draw (\i-1,1) rectangle (\i,2);
  \node at (\i-0.5,1.5) {$\v$};
}
\node[anchor=east] at (-0.3,0.5) {$R_1$};
\node[anchor=east] at (-0.3,1.5) {$R_2$};
\end{tikzpicture}
\caption{The positioning $\rho$.}
\label{fig:column-word-other}
\end{figure}
\end{example}

\subsection{Properties of columns}

In this subsection we prove several technical facts about columns that
will be needed later.

\begin{lemma}[Columns under an initial trivial segment]
\label{lem:trivial-column-identity}
Let $(T,\rho)$ be an aligned row array with
$T=(R_1,\dots,R_m)$, and fix $i\ge1$. If $\rho_2=\rho^{\mathrm{tr}}_2,\ \ldots,\
\rho_i=\rho^{\mathrm{tr}}_i$ (in particular, if
$\rho=\rho^{\mathrm{tr}}$), then $C(i,j)=j$ for every
$1\le j\le|R_i|$.
\end{lemma}

\begin{proof}
By Definition~\ref{def:column},
\[
\psi(k,j')=\bigl(k-1,\rho^{\mathrm{tr}}_k(j')\bigr)=(k-1,j'),
\qquad 2\le k\le i,\ 1\le j'\le|R_k|.
\]
Iterating from $(i,j)$ gives $\psi^{i-1}(i,j)=(1,j)$, i.e.\
$C(i,j)=j$.
\end{proof}

\begin{lemma}[Columns respect order within a row]
\label{lem:column-order}
Let $(T,\rho)$ be an aligned row array, $T=(R_1,\dots,R_m)$. Fix a row
$i$ and indices $1\le j<j'\le|R_i|$. Then $C(i,j)<C(i,j')$.
\end{lemma}

\begin{proof}
Induction on $i$. For $i=1$, $C(1,j)=j<j'=C(1,j')$ by definition. For
$i>1$, by Definition~\ref{def:column},
\[
C(i,j)=C(i-1,\rho_i(j)), \qquad C(i,j')=C(i-1,\rho_i(j')).
\]
By (P2), $\rho_i(j)<\rho_i(j')$, so the induction hypothesis applied to
row $i-1$ gives $C(i-1,\rho_i(j))<C(i-1,\rho_i(j'))$. Hence
$C(i,j)<C(i,j')$.
\end{proof}

\begin{lemma}[An equal entry in the row below sits to the right]
\label{lem:equal-entry-right}
Let $(T,\rho)$ be an aligned row array. Fix $i>1$ and columns
$C_r,C_{r'}$ such that $C_r(i)$ and $C_{r'}(i-1)$ are both defined.
If $C_r(i)=C_{r'}(i-1)$, then $r<r'$.
\end{lemma}

\begin{proof}
Let $j$ be the position with $(i,j)\in C_r$, so $R_{i,j}=C_r(i)$, and
let $j'$ be the position with $(i-1,j')\in C_{r'}$, so
$R_{i-1,j'}=C_{r'}(i-1)$; thus $R_{i,j}=R_{i-1,j'}$.

By (P1), $R_{i-1,\rho_i(j)}<R_{i,j}=R_{i-1,j'}$. Since $R_{i-1}$ is
weakly increasing, this rules out $j'\le\rho_i(j)$; so $j'>\rho_i(j)$.

Hence, by Definition~\ref{def:column} and
Lemma~\ref{lem:column-order} applied to row $i-1$,
$r=C(i,j)=C(i-1,\rho_i(j))<C(i-1,j')=r'$.
\end{proof}

\begin{corollary}[An equal entry in a lower row sits to the right]
\label{cor:equal-entry-right}
Let $(T,\rho)$ be an aligned row array, and let a letter $a$ occur in
row $i$, in column $C_r$, and in row $i'<i$, in column $C_{r'}$. If
$a$ occurs in every row between $i'$ and $i$, then $r<r'$.
\end{corollary}

\begin{proof}
Induction on $i-i'$; for $i-i'=1$ this is
Lemma~\ref{lem:equal-entry-right}. Otherwise, take any occurrence of
$a$ in row $i-1$, in a column $C_{r_1}$: by
Lemma~\ref{lem:equal-entry-right}, $r<r_1$, and by induction,
$r_1<r'$.
\end{proof}

\section{Column words are plactically equivalent}
\label{sec:col-plactic}

Recall from Section~\ref{sec:preliminaries} the Knuth relations (K1),
(K2) and plactic equivalence $\equiv_{\mathrm{plac}}$ on words over a
totally ordered alphabet. Throughout this section, we write $\equiv$
for $\equiv_{\mathrm{plac}}$: we use it often, and dropping the
subscript simplifies notation.

The goal of this section is to show that the column word of a row
array is, up to plactic (Knuth) equivalence, independent of the
positioning used to compute it (Theorem~\ref{thm:col-invariant},
Corollary~\ref{cor:col-well-defined}). We take this as a good sign
for the notion of positioning itself: the trivial positioning is not
privileged, and a positioning better suited to the problem at hand
may be used instead without changing the resulting plactic class.
This freedom is essential in Section~\ref{sec:catalan-image}, where a
non-trivial positioning is the convenient one.

\begin{lemma}[Sliding Lemma]
\label{lem:sliding}
Let $a_1<a_2<\cdots<a_\ell\le b$. Then
\[
a_\ell\cdots a_2a_1\,b \;\equiv\; a_\ell\,b\,a_{\ell-1}\cdots a_1.
\]
\end{lemma}

\begin{proof}
Induction on $\ell$; trivial for $\ell=1$. For $\ell>1$: since $a_1<\cdots<a_{\ell-1}\le b$,
the induction hypothesis applied to $a_1,\ldots,a_{\ell-1}$ and $b$ gives
$a_{\ell-1}\cdots a_1\,b\equiv a_{\ell-1}\,b\,a_{\ell-2}\cdots a_1$, so, prepending $a_\ell$,
\[
a_\ell\cdots a_1\,b \;\equiv\; a_\ell\,a_{\ell-1}\,b\,a_{\ell-2}\cdots a_1.
\]
Now apply (K1) to the first three letters, with $y=a_\ell,\ x=a_{\ell-1},\
z=b$: since $a_{\ell-1}<a_\ell\le b$, this gives $a_\ell a_{\ell-1}b\equiv a_\ell ba_{\ell-1}$, hence
\[
a_\ell\,a_{\ell-1}\,b\,a_{\ell-2}\cdots a_1 \;\equiv\; a_\ell\,b\,a_{\ell-1}\,a_{\ell-2}\cdots a_1. \qedhere
\]
\end{proof}

\begin{lemma}[Swap Lemma, base case $k=1$]
\label{lem:swap-base}
Let $\ell\ge1$. Let $a_1<\cdots<a_\ell$ and $b_1<\cdots<b_\ell$ be increasing
words with $a_i\le b_i$ for every $i$, and let $c>a_\ell$ and $c>b_\ell$.
Then
\[
c\,a_\ell\cdots a_1\,b_\ell\cdots b_1 \;\equiv\; a_\ell\cdots a_1\,c\,b_\ell\cdots b_1.
\]
\end{lemma}

\begin{proof}
Induction on $\ell$.

For $\ell=1$: the word $a_1cb_1$ satisfies $a_1\le b_1<c$, so (K2) gives
directly $ca_1b_1 \equiv a_1cb_1$.

For $\ell>1$, assume the lemma holds for $\ell-1$; explicitly, for any
increasing words $a'_1<\cdots<a'_{\ell-1}$ and $b'_1<\cdots<b'_{\ell-1}$ with
$a'_i\le b'_i$ for every $i$, and any $c'>a'_{\ell-1}$, $c'>b'_{\ell-1}$,
\[
c'\,a'_{\ell-1}\cdots a'_1\,b'_{\ell-1}\cdots b'_1 \;\equiv\;
a'_{\ell-1}\cdots a'_1\,c'\,b'_{\ell-1}\cdots b'_1.
\]
We invoke this induction hypothesis twice below, each time with its
own substitution for $a',b',c'$. Start from the word
$c\,a_\ell\cdots a_1\,b_\ell\cdots b_1$.

\emph{(a)} Since $a_\ell\le b_\ell$, the Sliding Lemma
(Lemma~\ref{lem:sliding}) gives
$a_\ell\cdots a_1\,b_\ell\equiv a_\ell\,b_\ell\,a_{\ell-1}\cdots a_1$, so, prepending $c$ and
appending $b_{\ell-1}\cdots b_1$,
\[
c\,a_\ell\cdots a_1\,b_\ell\cdots b_1 \;\equiv\;
c\,a_\ell\,b_\ell\,a_{\ell-1}\cdots a_1\,b_{\ell-1}\cdots b_1.
\]

\emph{(b)} Apply (K2) to the first three letters, valid since
$a_\ell\le b_\ell<c$. This gives $c\,a_\ell\,b_\ell\equiv a_\ell\,c\,b_\ell$, so
\[
c\,a_\ell\,b_\ell\,a_{\ell-1}\cdots a_1\,b_{\ell-1}\cdots b_1 \;\equiv\;
a_\ell\,c\,b_\ell\,a_{\ell-1}\cdots a_1\,b_{\ell-1}\cdots b_1.
\]

\emph{(c)} Apply the induction hypothesis to the subword
$b_\ell\,a_{\ell-1}\cdots a_1\,b_{\ell-1}\cdots b_1$, via
\[
a'_i:=a_i, \qquad b'_i:=b_i \qquad (i=1,\ldots,\ell-1), \qquad
c':=b_\ell,
\]
valid since $b_\ell$ exceeds every one of $a_{\ell-1},\ldots,a_1,b_{\ell-1},\ldots,b_1$
(as $b_\ell\ge a_\ell>a_{\ell-1}>\cdots>a_1$ and $b_\ell>b_{\ell-1}>\cdots>b_1$), and since
$a_i\le b_i$ for $i=1,\ldots,\ell-1$ holds by hypothesis. This
gives
\[
b_\ell\,a_{\ell-1}\cdots a_1\,b_{\ell-1}\cdots b_1 \;\equiv\;
a_{\ell-1}\cdots a_1\,b_\ell\,b_{\ell-1}\cdots b_1,
\]
so, prepending $a_\ell\,c$,
\[
a_\ell\,c\,b_\ell\,a_{\ell-1}\cdots a_1\,b_{\ell-1}\cdots b_1 \;\equiv\;
a_\ell\,c\,a_{\ell-1}\cdots a_1\,b_\ell\,b_{\ell-1}\cdots b_1.
\]

\emph{(d)} Apply the induction hypothesis once more, to the subword
$c\,a_{\ell-1}\cdots a_1\,b_\ell\cdots b_2$, via
\[
a'_i:=a_i, \qquad b'_i:=b_{i+1} \qquad (i=1,\ldots,\ell-1), \qquad c':=c,
\]
(so $b_1$ is left over). This is valid: $c$
exceeds every one of $a_{\ell-1},\ldots,a_1,b_\ell,\ldots,b_2$, and for
$i=1,\ldots,\ell-1$,
\[
b_{i+1} \;>\; b_i \;\ge\; a_i,
\]
using $b_i\ge a_i$ (by hypothesis) together with $b_{i+1}>b_i$;
so the required domination $b_{i+1}\ge a_i$ holds. This gives
\[
c\,a_{\ell-1}\cdots a_1\,b_\ell\cdots b_2 \;\equiv\;
a_{\ell-1}\cdots a_1\,c\,b_\ell\cdots b_2,
\]
so, prepending $a_\ell$ and appending the leftover $b_1$,
\[
a_\ell\,c\,a_{\ell-1}\cdots a_1\,b_\ell\cdots b_2\,b_1 \;\equiv\;
a_\ell\cdots a_1\,c\,b_\ell\cdots b_2\,b_1.
\]

Chaining (a)--(d) gives $c\,a_\ell\cdots a_1\,b_\ell\cdots b_1 \equiv
a_\ell\cdots a_1\,c\,b_\ell\cdots b_1$. \qedhere
\end{proof}

\begin{lemma}[Swap Lemma]
\label{lem:swap}
Let $\ell,k\ge1$. Let $a_1<\cdots<a_\ell$ and $b_1<\cdots<b_\ell$ be increasing
words with $a_i\le b_i$ for every $i$, and let $c_1<\cdots<c_k$
satisfy $c_1>a_\ell$ and $c_1>b_\ell$ (so every $c_s$ exceeds every
$a_i,b_i$). Then
\[
c_k\cdots c_1\;a_\ell\cdots a_1\;b_\ell\cdots b_1 \;\equiv\;
a_\ell\cdots a_1\;c_k\cdots c_1\;b_\ell\cdots b_1.
\]
\end{lemma}

\begin{proof}
Induction on $k$. The case $k=1$ is exactly Lemma~\ref{lem:swap-base}.
Assume the lemma holds for $k-1$; explicitly, for any increasing
$a'_1<\cdots<a'_\ell$ and $b'_1<\cdots<b'_\ell$ with $a'_i\le b'_i$ for every $i$,
and any $c'_1<\cdots<c'_{k-1}$ with $c'_1>a'_\ell$, $c'_1>b'_\ell$,
\[
c'_{k-1}\cdots c'_1\,a'_\ell\cdots a'_1\,b'_\ell\cdots b'_1 \;\equiv\;
a'_\ell\cdots a'_1\,c'_{k-1}\cdots c'_1\,b'_\ell\cdots b'_1.
\]
We show the lemma for $k$ in two steps: first we push $c_1$ inward,
and then $c_k\cdots c_2$:
\[
c_k\cdots c_2\,c_1\,a_\ell\cdots a_1\,b_\ell\cdots b_1
\;\overset{(1)}{\equiv}\;
c_k\cdots c_2\,a_\ell\cdots a_1\,c_1\,b_\ell\cdots b_1
\;\overset{(2)}{\equiv}\;
a_\ell\cdots a_1\,c_k\cdots c_2\,c_1\,b_\ell\cdots b_1.
\]

\emph{Step (1).} Apply Lemma~\ref{lem:swap-base} to the suffix
$c_1\,a_\ell\cdots a_1\,b_\ell\cdots b_1$, with $c:=c_1$; this is valid
since $c_1>a_\ell$ and $c_1>b_\ell$.

\emph{Step (2).} Apply the induction hypothesis to the prefix
$c_k\cdots c_2\,a_\ell\cdots a_1\,c_1\,b_\ell\cdots b_2$ (so $b_1$ is left
over), via
\[
a'_i:=a_i\ (i=1,\ldots,\ell), \qquad
b'_i:=b_{i+1}\ (i=1,\ldots,\ell-1),\ \ b'_\ell:=c_1, \qquad
c'_i:=c_{i+1}\ (i=1,\ldots,k-1).
\]
This is valid: $b'_1<\cdots<b'_\ell$ is strictly increasing (as $b_\ell<c_1$);
$c'_1=c_2$ exceeds both $a'_\ell=a_\ell$ and $b'_\ell=c_1$ (as $c_2>c_1>a_\ell$); and
$b'_i\ge a'_i$ for every $i$ -- for $i=1,\ldots,\ell-1$,
$b'_i=b_{i+1}>b_i\ge a_i=a'_i$, and $b'_\ell=c_1\ge a_\ell=a'_\ell$ by hypothesis.
\end{proof}

\begin{definition}[Excess]
\label{def:excess}
Let $\rho$ be a positioning on $T=(R_1,\ldots,R_m)$. The \emph{excess}
of $\rho$ is
\[
\mathrm{exc}(\rho) := \sum_{i=2}^{m}\sum_{j=1}^{|R_i|}
\bigl(\rho_i(j)-j\bigr).
\]
By Corollary~\ref{cor:rho-dominates-id}, $\rho_i(j)\ge j$
always, so $\mathrm{exc}(\rho)\ge0$, with equality iff
$\rho=\rho^{\mathrm{tr}}$.
\end{definition}

\begin{lemma}[Elementary covering step]
\label{lem:covering-step}
Let $\rho\ne\rho^{\mathrm{tr}}$ be a positioning on
$T=(R_1,\ldots,R_m)$. Since $\rho\ne\rho^{\mathrm{tr}}$, some row
$i>1$ has $\rho_i\ne\rho^{\mathrm{tr}}_i$; let $i_0$ be
the smallest such row (the deviating row closest to the bottom). Note
that $i_0\ge2$. Since
$\rho_{i_0}(j)\ge j$ for every $j$ (Corollary~\ref{cor:rho-dominates-id}),
$\rho_{i_0}\ne\rho^{\mathrm{tr}}_{i_0}$ means $\rho_{i_0}(j)>j$ for some $j$;
let $j_0$ be the smallest such $j$, and set $t:=\rho_{i_0}(j_0)$. Define
$\rho'$ by $\rho'_{i_0}(j_0):=t-1$ and $\rho':=\rho$ elsewhere. Then:
\begin{enumerate}
\item $t>j_0$, and no $j\ne j_0$ in row $i_0$ satisfies
$\rho_{i_0}(j)=t-1$ or $\rho'_{i_0}(j)=t$;
\item $\rho'$ is a positioning on $T$ with
$\mathrm{exc}(\rho')=\mathrm{exc}(\rho)-1$
(Definition~\ref{def:excess});
\item $\rho'_i=\rho^{\mathrm{tr}}_i$ for every $i<i_0$.
\end{enumerate}
\end{lemma}

\begin{proof}
Recall $t=\rho_{i_0}(j_0)>j_0$ directly by the choice of $j_0$. Also
note that $\rho_{i_0}(j)=j$ for every $j<j_0$: minimality of $j_0$
rules out $\rho_{i_0}(j)>j$, so $\rho_{i_0}(j)\ge j$
(Corollary~\ref{cor:rho-dominates-id}) forces equality.

For (1), it remains to check that no $j\ne j_0$ satisfies $\rho_{i_0}(j)=t-1$ or
$\rho'_{i_0}(j)=t$: for $j<j_0$, $\rho_{i_0}(j)=j<j_0\le t-1$ (using
the note above and $t>j_0$ just shown), so
$\rho_{i_0}(j)\ne t-1$; for $j>j_0$, strict monotonicity of
$\rho_{i_0}$ gives $\rho_{i_0}(j)>\rho_{i_0}(j_0)=t>t-1$, so again
$\rho_{i_0}(j)\ne t-1$. This rules out $\rho_{i_0}(j)=t-1$ for every
$j\ne j_0$. Likewise, since $\rho'_{i_0}(j)=\rho_{i_0}(j)$ for every
$j\ne j_0$, and $\rho_{i_0}$ is injective (being strictly
increasing) with $\rho_{i_0}(j_0)=t$, no $j\ne j_0$ has
$\rho_{i_0}(j)=t$, hence none has $\rho'_{i_0}(j)=t$ either.

We now prove (2) by checking that $\rho'$ satisfies the two
defining properties (P1), (P2) of a positioning
(Definition~\ref{def:positioning}).

(P1): for $i\ne i_0$, $\rho'_i=\rho_i$, so (P1) holds
there because it holds for $\rho$; at $i=i_0$, likewise for $j\ne
j_0$, $\rho'_{i_0}(j)=\rho_{i_0}(j)$, so (P1) holds there too because it
holds for $\rho$, and it remains to check $j=j_0$. Since $R_{i_0-1}$
is weakly increasing and $\rho'_{i_0}(j_0)<\rho_{i_0}(j_0)$,
\[
R_{i_0-1,\,\rho'_{i_0}(j_0)} \;\le\; R_{i_0-1,\,\rho_{i_0}(j_0)}.
\]
And by (P1) for $\rho$,
\[
R_{i_0-1,\,\rho_{i_0}(j_0)} \;<\; R_{i_0,j_0}.
\]
So $R_{i_0-1,\,\rho'_{i_0}(j_0)}<R_{i_0,j_0}$, i.e.\ (P1) holds for
$\rho'$ at $(i_0,j_0)$ as well.

(P2): for $i\ne i_0$, $\rho'_i=\rho_i$, so (P2) holds
there because it holds for $\rho$; at $i=i_0$, for adjacent pairs not
involving $j_0$, $\rho'$ agrees with $\rho$, so monotonicity holds
there too because it holds for $\rho$, and it remains to check the
pairs adjacent to $j_0$. On the left, if $j_0>1$ then
$\rho'_{i_0}(j_0)=t-1\ge j_0>j_0-1=\rho_{i_0}(j_0-1)=\rho'_{i_0}(j_0-1)$
(using $t>j_0$ from (1), and $\rho_{i_0}(j_0-1)=j_0-1$ by minimality of $j_0$, as noted at the
start of the proof); on the right, if $j_0<|R_{i_0}|$ then
$\rho'_{i_0}(j_0)<t=\rho_{i_0}(j_0)<\rho_{i_0}(j_0+1)=\rho'_{i_0}(j_0+1)$. So
$\rho'_{i_0}$ is strictly increasing.

So $\rho'$ is a positioning, proving (2);
$\mathrm{exc}(\rho')=\mathrm{exc}(\rho)-1$ is immediate from the
definition. Item (3) holds since $i_0$ is minimal and $\rho'$ agrees
with $\rho$ outside $(i_0,j_0)$.
\end{proof}

\begin{lemma}[The covering step's column]
\label{lem:covering-step-column}
With $\rho,\rho',i_0,j_0,t$ as in Lemma~\ref{lem:covering-step}
(so $i_0\ge2$), let
\[
a_i:=R_{i,\,t-1}, \qquad b_i:=R_{i,\,t} \qquad
(i=1,\ldots,i_0-1).
\]
Then:
\begin{enumerate}
\item Under both $\rho$ and $\rho'$, $C_{t-1}(i)=a_i$ and $C_t(i)=b_i$ for $i=1,\ldots,i_0-1$;
\item $(i_0,j_0)$ lies in $C_t$ under $\rho$, and in $C_{t-1}$
under $\rho'$.
\end{enumerate}
\end{lemma}

\begin{proof}
By minimality of $i_0$, $\rho_2=\rho^{\mathrm{tr}}_2,\ \ldots,\
\rho_{i_0-1}=\rho^{\mathrm{tr}}_{i_0-1}$; by
Lemma~\ref{lem:covering-step}(3), the same holds for $\rho'$. So
Lemma~\ref{lem:trivial-column-identity} gives $C(i,j)=j$ under
both $\rho$ and $\rho'$, for every row $i<i_0$.

Applying this with index $t-1$, at rows $i=1,\ldots,i_0-1$, gives
$(i,t-1)\in C_{t-1}$, i.e.\ $C_{t-1}(i)=R_{i,t-1}=a_i$, for every such
$i$. Applying it with index $t$ likewise gives $C_t(i)=b_i$. This
proves (1).

For (2): $\rho_{i_0}(j_0)=t$, so $(i_0,j_0)$ lies in the same column
as $(i_0-1,t)$ (Definition~\ref{def:column}), which is $C_t$ by the
fact above; likewise
$\rho'_{i_0}(j_0)=t-1$ puts $(i_0,j_0)$ in the same column as
$(i_0-1,t-1)$, i.e.\ $C_{t-1}$.
\end{proof}

\begin{lemma}[Local effect of a covering step]
\label{lem:local-effect}
With $\rho,\rho',i_0,j_0,t$ as in Lemma~\ref{lem:covering-step}, let
$a_i,b_i$ ($i=1,\ldots,i_0-1$) be as in
Lemma~\ref{lem:covering-step-column}. By that lemma, $(i_0,j_0)$ lies
in $C_t$ under $\rho$ and in $C_{t-1}$ under $\rho'$; since
$\rho,\rho'$ agree above row $i_0$, these two columns share the same
entries from row $i_0$ up to their topmost row -- let
$c_1<\cdots<c_k$ denote them (so $k\ge1$ and $c_1=R_{i_0,j_0}$),
diverging only below row $i_0$, as described in (2) below (see
Figure~\ref{fig:covering-step}). Then:
\begin{enumerate}
\item $a_1<\cdots<a_{i_0-1}$, $b_1<\cdots<b_{i_0-1}$, $a_i\le b_i$
for every $i$, and every $c_s$ exceeds every $a_i,b_i$;
\item Under $\rho$, we have $w(C_{t-1})=a_{i_0-1}\cdots a_1$ and
$w(C_t)=c_k\cdots c_1\,b_{i_0-1}\cdots b_1$. Under $\rho'$,
$w(C_{t-1})=c_k\cdots c_1\,a_{i_0-1}\cdots a_1$ and
$w(C_t)=b_{i_0-1}\cdots b_1$. Every
other column of $(T,\rho)$ equals the corresponding column of
$(T,\rho')$.
\end{enumerate}
\end{lemma}

\begin{figure}[h]
\centering
\begin{tikzpicture}[scale=0.85, every node/.style={font=\small}]

\draw (0,-3) rectangle (1,0);
\draw (0,-2) -- (1,-2); \draw (0,-1) -- (1,-1);
\node at (0.5,-0.5) {$*$}; \node at (0.5,-1.5) {$\vdots$}; \node at (0.5,-2.5) {$*$};

\draw (1,-3) rectangle (2,3);
\draw (1,-2) -- (2,-2); \draw (1,-1) -- (2,-1); \draw (1,0) -- (2,0);
\draw (1,1) -- (2,1); \draw (1,2) -- (2,2);
\node at (1.5,2.5) {$*$}; \node at (1.5,1.5) {$*$}; \node at (1.5,0.5) {$*$};
\node at (1.5,-0.5) {$*$}; \node at (1.5,-1.5) {$\vdots$}; \node at (1.5,-2.5) {$*$};

\draw (2,-3) rectangle (3,0);
\draw (2,-2) -- (3,-2); \draw (2,-1) -- (3,-1);
\node at (2.5,-0.5) {$a_{i_0-1}$}; \node at (2.5,-1.5) {$\vdots$}; \node at (2.5,-2.5) {$a_1$};

\draw (3,-3) rectangle (4,4);
\draw (3,-2) -- (4,-2); \draw (3,-1) -- (4,-1); \draw (3,0) -- (4,0);
\draw (3,1) -- (4,1); \draw (3,2) -- (4,2); \draw (3,3) -- (4,3);
\node at (3.5,3.5) {$c_k$}; \node at (3.5,2.5) {$\vdots$};
\node at (3.5,1.5) {$c_2$}; \node at (3.5,0.5) {$c_1$};
\node at (3.5,-0.5) {$b_{i_0-1}$}; \node at (3.5,-1.5) {$\vdots$}; \node at (3.5,-2.5) {$b_1$};

\draw (4,-3) rectangle (5,1);
\draw (4,-2) -- (5,-2); \draw (4,-1) -- (5,-1); \draw (4,0) -- (5,0);
\node at (4.5,0.5) {$*$}; \node at (4.5,-0.5) {$*$};
\node at (4.5,-1.5) {$\vdots$}; \node at (4.5,-2.5) {$*$};

\draw (5,-3) rectangle (6,2);
\draw (5,-2) -- (6,-2); \draw (5,-1) -- (6,-1); \draw (5,0) -- (6,0); \draw (5,1) -- (6,1);
\node at (5.5,1.5) {$*$}; \node at (5.5,0.5) {$*$}; \node at (5.5,-0.5) {$*$};
\node at (5.5,-1.5) {$\vdots$}; \node at (5.5,-2.5) {$*$};

\node[anchor=south east] at (3,4.3) {$C_{t-1}$};
\node[anchor=south west] at (3,4.3) {$C_t$};

\node[anchor=east] at (-0.2,0.5) {$R_{i_0}$};
\node[anchor=east] at (-0.2,-0.5) {$R_{i_0-1}$};
\node[anchor=east] at (-0.2,-2.5) {$R_1$};

\node at (3,-3.6) {$(T,\rho)$};

\end{tikzpicture}
\hspace{1.5cm}
\begin{tikzpicture}[scale=0.85, every node/.style={font=\small}]

\draw (0,-3) rectangle (1,0);
\draw (0,-2) -- (1,-2); \draw (0,-1) -- (1,-1);
\node at (0.5,-0.5) {$*$}; \node at (0.5,-1.5) {$\vdots$}; \node at (0.5,-2.5) {$*$};

\draw (1,-3) rectangle (2,3);
\draw (1,-2) -- (2,-2); \draw (1,-1) -- (2,-1); \draw (1,0) -- (2,0);
\draw (1,1) -- (2,1); \draw (1,2) -- (2,2);
\node at (1.5,2.5) {$*$}; \node at (1.5,1.5) {$*$}; \node at (1.5,0.5) {$*$};
\node at (1.5,-0.5) {$*$}; \node at (1.5,-1.5) {$\vdots$}; \node at (1.5,-2.5) {$*$};

\draw (2,-3) rectangle (3,4);
\draw (2,-2) -- (3,-2); \draw (2,-1) -- (3,-1); \draw (2,0) -- (3,0);
\draw (2,1) -- (3,1); \draw (2,2) -- (3,2); \draw (2,3) -- (3,3);
\node at (2.5,3.5) {$c_k$}; \node at (2.5,2.5) {$\vdots$};
\node at (2.5,1.5) {$c_2$}; \node at (2.5,0.5) {$c_1$};
\node at (2.5,-0.5) {$a_{i_0-1}$}; \node at (2.5,-1.5) {$\vdots$}; \node at (2.5,-2.5) {$a_1$};

\draw (3,-3) rectangle (4,0);
\draw (3,-2) -- (4,-2); \draw (3,-1) -- (4,-1);
\node at (3.5,-0.5) {$b_{i_0-1}$}; \node at (3.5,-1.5) {$\vdots$}; \node at (3.5,-2.5) {$b_1$};

\draw (4,-3) rectangle (5,1);
\draw (4,-2) -- (5,-2); \draw (4,-1) -- (5,-1); \draw (4,0) -- (5,0);
\node at (4.5,0.5) {$*$}; \node at (4.5,-0.5) {$*$};
\node at (4.5,-1.5) {$\vdots$}; \node at (4.5,-2.5) {$*$};

\draw (5,-3) rectangle (6,2);
\draw (5,-2) -- (6,-2); \draw (5,-1) -- (6,-1); \draw (5,0) -- (6,0); \draw (5,1) -- (6,1);
\node at (5.5,1.5) {$*$}; \node at (5.5,0.5) {$*$}; \node at (5.5,-0.5) {$*$};
\node at (5.5,-1.5) {$\vdots$}; \node at (5.5,-2.5) {$*$};

\node[anchor=south east] at (3,4.3) {$C_{t-1}$};
\node[anchor=south west] at (3,4.3) {$C_t$};

\node[anchor=east] at (-0.2,0.5) {$R_{i_0}$};
\node[anchor=east] at (-0.2,-0.5) {$R_{i_0-1}$};
\node[anchor=east] at (-0.2,-2.5) {$R_1$};

\node at (3,-3.6) {$(T,\rho')$};

\end{tikzpicture}
\caption{Schematic view of $(T,\rho)$ and $(T,\rho')$.}
\label{fig:covering-step}
\end{figure}
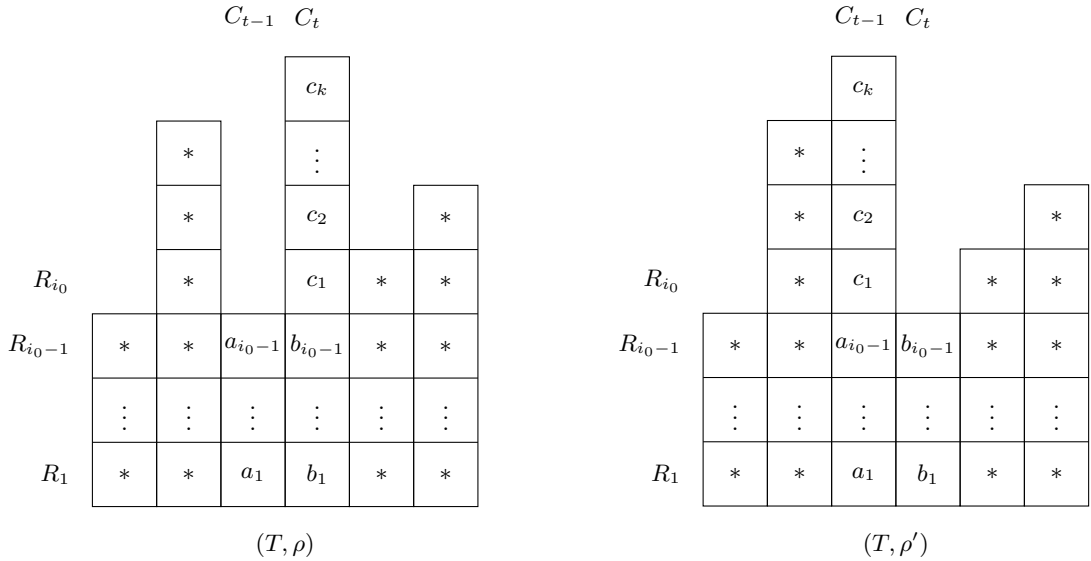

\begin{proof}
By Lemma~\ref{lem:covering-step-column}(1), $a_1,\ldots,a_{i_0-1}$ and
$b_1,\ldots,b_{i_0-1}$ are exactly the entries of $C_{t-1}$ and
$C_t$ below row $i_0$ (under either $\rho$ or $\rho'$); they are
strictly increasing by (P1). For each $i$, $a_i=R_{i,t-1}\le R_{i,t}=b_i$ since
$R_i$ is weakly increasing and $t-1<t$. By (P1) applied to
$\rho$ at $(i_0,j_0)$ (where $\rho_{i_0}(j_0)=t$), $c_1=R_{i_0,j_0}>
R_{i_0-1,t}=b_{i_0-1}\ge b_i$ for all $i$; likewise by (P1) for $\rho'$
(where $\rho'_{i_0}(j_0)=t-1$), $c_1>R_{i_0-1,t-1}=a_{i_0-1}\ge a_i$
for all $i$; since $c_s\ge c_1$ for every $s\ge1$, every $c_s$ exceeds
every $a_i,b_i$.

Recall that the column of $(i_0,j_0)$, restricted to rows $\ge i_0$,
is $c_k\cdots c_1$ by definition.

\emph{Under $\rho$.} The position $(i_0,j_0)$ lies in $C_t$
(Lemma~\ref{lem:covering-step-column}(2)), and no other position of
row $i_0$ lies in $C_{t-1}$ (Lemma~\ref{lem:covering-step}(1)). So
$C_{t-1}$ has nothing at row $i_0$ or above, i.e.\
$w(C_{t-1})=a_{i_0-1}\cdots a_1$, and $w(C_t)=c_k\cdots
c_1\,b_{i_0-1}\cdots b_1$.

\emph{Under $\rho'$.} The position $(i_0,j_0)$ lies in $C_{t-1}$
(Lemma~\ref{lem:covering-step-column}(2)), and no other position of
row $i_0$ lies in $C_t$ (Lemma~\ref{lem:covering-step}(1)). So
$w(C_{t-1})=c_k\cdots c_1\,a_{i_0-1}\cdots a_1$ and
$w(C_t)=b_{i_0-1}\cdots b_1$.

Since $\rho,\rho'$ agree everywhere except at $(i_0,j_0)$, every
column other than $C_{t-1},C_t$ is unaffected.
\end{proof}

\begin{theorem}[The plactic class of the column word is independent of positioning]
\label{thm:col-invariant}
Let $T$ be a row array admitting a positioning. For every positioning
$\rho$ on $T$,
\[
\mathrm{col}(T,\rho) \;\equiv\; \mathrm{col}(T,\rho^{\mathrm{tr}}).
\]
\end{theorem}

\begin{proof}
Induction on $\mathrm{exc}(\rho)$. If $\mathrm{exc}(\rho)=0$
then $\rho=\rho^{\mathrm{tr}}$ and there is nothing to prove.
Otherwise let $\rho'$ be obtained from $\rho$ by the elementary
covering step of Lemma~\ref{lem:covering-step}, so
$\mathrm{exc}(\rho')=\mathrm{exc}(\rho)-1$. By
Lemma~\ref{lem:local-effect}, $\mathrm{col}(T,\rho)$ and
$\mathrm{col}(T,\rho')$ agree on every column except $C_{t-1},C_t$, where
they read, respectively,
\[
\mathrm{col}(T,\rho):\quad \cdots\ a_{i_0-1}\cdots a_1\ \ c_k\cdots c_1\,b_{i_0-1}\cdots b_1\ \cdots
\]
\[
\mathrm{col}(T,\rho'):\quad \cdots\ c_k\cdots c_1\,a_{i_0-1}\cdots a_1\ \ b_{i_0-1}\cdots b_1\ \cdots,
\]
with $a_1<\cdots<a_{i_0-1}$, $b_1<\cdots<b_{i_0-1}$, $a_i\le b_i$, and
every $c_s$ exceeding every $a_i,b_i$. By the Swap Lemma
(Lemma~\ref{lem:swap}), with $\ell:=i_0-1$ and the $a,b,c$ just
introduced,
\[
c_k\cdots c_1\,a_{i_0-1}\cdots a_1\,b_{i_0-1}\cdots b_1 \;\equiv\;
a_{i_0-1}\cdots a_1\,c_k\cdots c_1\,b_{i_0-1}\cdots b_1.
\]
As Knuth equivalence is a congruence, prepending and appending the
(identical) remaining columns gives $\mathrm{col}(T,\rho')\equiv
\mathrm{col}(T,\rho)$. By the induction hypothesis,
$\mathrm{col}(T,\rho')\equiv\mathrm{col}(T,\rho^{\mathrm{tr}})$, so
$\mathrm{col}(T,\rho)\equiv\mathrm{col}(T,\rho^{\mathrm{tr}})$ by
transitivity.
\end{proof}

\begin{corollary}
\label{cor:col-well-defined}
For any two positionings $\rho,\rho'$ on a row array $T$,
$\mathrm{col}(T,\rho)\equiv\mathrm{col}(T,\rho')$.
\end{corollary}

\begin{proof}
By Lemma~\ref{lem:positioning-implies-trivial}, $\rho^{\mathrm{tr}}$
exists, so Theorem~\ref{thm:col-invariant} applies to both $\rho$ and
$\rho'$: $\mathrm{col}(T,\rho)\equiv\mathrm{col}(T,\rho^{\mathrm{tr}})
\equiv\mathrm{col}(T,\rho')$.
\end{proof}

\section{Tight positioning}
\label{sec:tight}

\subsection{Definition of tight positioning}

\begin{definition}[Tight positioning]
\label{def:tight-positioning}
Let $R_{i-1}$ and $R_i$ be two rows. Define, whenever the set on the
right is nonempty,
\[
\rho^{\mathrm{tight}}_i(j) \;:=\; \max\bigl\{\, d : R_{i-1,d} < R_{i,j} \,\bigr\};
\]
this need not be defined for every $j$. Whenever it is defined, we
have $R_{i-1,\,\rho^{\mathrm{tight}}_i(j)}<R_{i,j}$, so (P1) holds
immediately.

Let $T=(R_1,\ldots,R_m)$ be a row array. If, for every $i=2,\ldots,m$
and $j=1,\ldots,|R_i|$, $\rho^{\mathrm{tight}}_i(j)$ is defined and
$\rho^{\mathrm{tight}}_i$ is, in addition, strictly increasing, i.e.\
satisfies (P2), then the resulting positioning $\rho^{\mathrm{tight}}$
on $T$ (Definition~\ref{def:positioning}) is called \emph{the tight
positioning} on $T$, and the aligned row array $(T,\rho^{\mathrm{tight}})$
is called \emph{the tight alignment} of $T$.
\end{definition}

\begin{example}[Tight positioning, pictured]
\label{ex:tight-picture}
Let $R_1=(1,2,4,6,8)$, $R_2=(2,5,7)$, and $R_3=(6)$. Then
\[
\rho^{\mathrm{tight}}_2(1) = \max\{d:R_{1,d}<2\} = 1, \qquad
\rho^{\mathrm{tight}}_2(2) = \max\{d:R_{1,d}<5\} = 3,
\]
\[
\rho^{\mathrm{tight}}_2(3) = \max\{d:R_{1,d}<7\} = 4, \qquad
\rho^{\mathrm{tight}}_3(1) = \max\{d:R_{2,d}<6\} = 2,
\]
each strictly increasing, so the tight positioning exists on
$(R_1,R_2,R_3)$: each entry is pushed as far right as possible while
still sitting above a strictly smaller entry of the row below.
Figure~\ref{fig:tight-picture} draws the resulting tight alignment.
\begin{figure}[H]
\centering
\begin{tikzpicture}[scale=0.9, every node/.style={font=\small}]
\foreach \i/\v in {1/1,2/2,3/4,4/6,5/8} {
  \draw (\i-1,0) rectangle (\i,1);
  \node at (\i-0.5,0.5) {$\v$};
}
\foreach \i/\v in {1/2,3/5,4/7} {
  \draw (\i-1,1) rectangle (\i,2);
  \node at (\i-0.5,1.5) {$\v$};
}
\foreach \i/\v in {3/6} {
  \draw (\i-1,2) rectangle (\i,3);
  \node at (\i-0.5,2.5) {$\v$};
}
\node[anchor=east] at (-0.3,0.5) {$R_1$};
\node[anchor=east] at (-0.3,1.5) {$R_2$};
\node[anchor=east] at (-0.3,2.5) {$R_3$};
\end{tikzpicture}
\caption{The tight alignment of Example~\ref{ex:tight-picture}.}
\label{fig:tight-picture}
\end{figure}
\end{example}

\begin{example}[The tight positioning need not exist]
\label{ex:tight-fails}
Let $R_1 = (1,2,3,4)$ and $R_2 = (5,6)$. Then
\[
\rho^{\mathrm{tight}}_2(1) = \max\{d : R_{1,d} < 5\} = 4, \qquad
\rho^{\mathrm{tight}}_2(2) = \max\{d : R_{1,d} < 6\} = 4,
\]
so $\rho^{\mathrm{tight}}_2(1) = \rho^{\mathrm{tight}}_2(2)$: condition
(P2) fails, so the tight positioning does not exist on
$(R_1,R_2)$. A positioning nonetheless exists on it, e.g.\
$\rho_2(1)=3,\ \rho_2(2)=4$, pictured in
Figure~\ref{fig:tight-fails-picture}.
\begin{figure}[H]
\centering
\begin{tikzpicture}[scale=0.9, every node/.style={font=\small}]
\foreach \i/\v in {1/1,2/2,3/3,4/4} {
  \draw (\i-1,0) rectangle (\i,1);
  \node at (\i-0.5,0.5) {$\v$};
}
\foreach \i/\v in {3/5,4/6} {
  \draw (\i-1,1) rectangle (\i,2);
  \node at (\i-0.5,1.5) {$\v$};
}
\node[anchor=east] at (-0.3,0.5) {$R_1$};
\node[anchor=east] at (-0.3,1.5) {$R_2$};
\end{tikzpicture}
\caption{A (non-tight) alignment of $(R_1,R_2)$, from
Example~\ref{ex:tight-fails}.}
\label{fig:tight-fails-picture}
\end{figure}
\end{example}

\begin{corollary}[Tight positioning is maximal]
Whenever the tight positioning exists on $(R_{i-1},R_i)$, it is the
maximum, among all positionings on it, in the order of
Definition~\ref{def:positioning-order}: for every positioning $\rho$
on $(R_{i-1},R_i)$ and every $j$,
\[
\rho_i(j) \;\le\; \rho^{\mathrm{tight}}_i(j).
\]
\end{corollary}

\begin{proof}
Let $\rho$ be any positioning on $(R_{i-1},R_i)$ and fix $j$. By (P1),
$R_{i-1,\,\rho_i(j)} < R_{i,j}$, so $\rho_i(j)$ belongs to the set
$\{d : R_{i-1,d} < R_{i,j}\}$ whose maximum defines
$\rho^{\mathrm{tight}}_i(j)$.
Hence $\rho_i(j)\le\rho^{\mathrm{tight}}_i(j)$.
\end{proof}

\subsection{Tight positioning in an \texorpdfstring{$N$}{N}-tableau}

\begin{proposition}[Tight positioning under row inclusion]
\label{prop:tight-exists-for-Ntableau}
Let $R_{i-1}$ and $R_i$ be rows with strictly increasing entries such
that $R_i \subseteq R_{i-1}$ and $\min R_{i-1} < \min R_i$. For
$j=1,\ldots,|R_i|$, let $\tau_i(j)$ be the unique index with
$R_{i-1,\,\tau_i(j)}=R_{i,j}$. Then $\rho^{\mathrm{tight}}_i(j)$ is
defined for every $j$, and
\[
\rho^{\mathrm{tight}}_i(j) \;=\; \tau_i(j)-1;
\]
in particular, the tight positioning exists on $(R_{i-1},R_i)$.
\end{proposition}

\begin{proof}
Since $R_{i-1}$ is strictly increasing, $R_{i-1,d}<R_{i-1,\,\tau_i(j)}=
R_{i,j}$ iff $d<\tau_i(j)$. Moreover $\tau_i(j)\ge2$: since
$R_{i-1,\,\tau_i(j)}=R_{i,j}\ge\min R_i>\min R_{i-1}=R_{i-1,1}$, we
have $\tau_i(j)\ne1$. Hence
\[
\rho^{\mathrm{tight}}_i(j) \;=\; \max\{d : R_{i-1,d}<R_{i,j}\} \;=\;
\max\{1,\ldots,\tau_i(j)-1\} \;=\; \tau_i(j)-1.
\]
(P1) always holds for $\rho^{\mathrm{tight}}$
(Definition~\ref{def:tight-positioning}). For (P2), if $j<j'$ then
$R_{i-1,\,\tau_i(j)}=R_{i,j}<R_{i,j'}=R_{i-1,\,\tau_i(j')}$, so
$\tau_i(j)<\tau_i(j')$ as $R_{i-1}$ is increasing; hence
$\rho^{\mathrm{tight}}_i$ is strictly increasing.
\end{proof}

Recall that an $N$-tableau is a finite sequence of nonempty rows
$R_1,\ldots,R_m$ such that (i) each $R_i$ is a finite strictly
increasing sequence of elements of $\mathbf{n}$, (ii) $R_i\subseteq R_{i-1}$ for
every $1<i\le m$, and (iii) $\min R_1<\min R_2<\cdots<\min R_m$. This
is the form that sits directly on our row-array framework.

\begin{corollary}
\label{cor:tight-exists-for-Ntableau}
Every $N$-tableau admits a tight positioning: apply
Proposition~\ref{prop:tight-exists-for-Ntableau} to every adjacent
pair $(R_{i-1},R_i)$, $i=2,\dots,m$, using $R_i \subseteq R_{i-1}$ and
$\min R_{i-1} < \min R_i$, both from the definition of $N$-tableau --
(ii) and (iii) respectively.
\end{corollary}

\begin{remark}
\label{rem:positioning-gives-minima}
Condition (iii) of an $N$-tableau is automatic once some positioning
exists. Indeed, a row array with nonempty rows admitting a positioning
also admits the trivial positioning $\rho^{\mathrm{tr}}$
(Lemma~\ref{lem:positioning-implies-trivial}), i.e.\ it is an SSYT
(Definition~\ref{def:ssyt-as-aligned}), and (P1) for
$\rho^{\mathrm{tr}}$ at $j=1$ gives
$\min R_{i-1}=R_{i-1,1}<R_{i,1}=\min R_i$: the first column of an SSYT
strictly increases.
\end{remark}

\begin{corollary}[Higher means further left]
\label{cor:higher-left}
Let $(T,\rho)$ be an aligned row array with $T$ an $N$-tableau, and
let a letter $a$ occur in row $i$, in column $C_r$, and in row $i'$, in
column $C_{r'}$. Then $r<r'$ if and only if $i'<i$.
\end{corollary}

\begin{proof}
By containment, $a$ occurs in every row below row $\max(i,i')$, and,
rows being strictly increasing, at most once in each row. So if
$i'<i$, Corollary~\ref{cor:equal-entry-right} gives $r<r'$; if
$i'>i$, it gives $r'<r$; and if $i'=i$, the two occurrences coincide
and $r=r'$.
\end{proof}

We end this subsection with a property of the tight positioning of
an $N$-tableau that will be used repeatedly in
Sections~\ref{sec:reduction} and~\ref{sec:catalan-image}.

\begin{lemma}[Successor sits directly above]
\label{lem:successor-above}
Let $T=(R_1,\ldots,R_m)$ be an $N$-tableau equipped with its tight
positioning $\rho=\rho^{\mathrm{tight}}$
(Corollary~\ref{cor:tight-exists-for-Ntableau}). Suppose $1\le k<m$ and
$R_{k,q_1}+1=R_{k+1,q_2}$. Then
\[
\rho_{k+1}(q_2) = q_1,
\]
i.e.\ $\psi(k+1,q_2)=(k,q_1)$: the entries $R_{k+1,q_2}$ (in row
$k+1$) and $R_{k,q_1}$ (in row $k$) lie in the same column of
$(T,\rho)$, one directly above the other.
\end{lemma}

\begin{proof}
By definition, $\rho_{k+1}(q_2)=\max\{\,d : R_{k,d}<R_{k+1,q_2}\,\}$.
Every element $R_{k,d}$ satisfying $R_{k,d}<R_{k+1,q_2}=R_{k,q_1}+1$
satisfies $R_{k,d}\le R_{k,q_1}$ (both are integers), and $R_{k,q_1}$
itself is attained, at position $q_1$. Hence $q_1$ is the maximum,
i.e.\ $\rho_{k+1}(q_2)=q_1$. The rest is immediate from the
definition of $\psi$ (Definition~\ref{def:column}).
\end{proof}

\section{Full \texorpdfstring{$N$}{N}-tableau}
\label{sec:reduction}

\subsection{Downward closed restrictions}

We define the restriction of an alignment to a downward closed set of
cells and record a few of its properties, in preparation for the
definition of the full core.

\begin{definition}[Restriction]
\label{def:restriction}
Let $(T,\rho)$ be an aligned row array, $T=(R_1,\dots,R_m)$, and let
$S$ be a set of cells of $T$ that is \emph{downward closed along
columns}: whenever $(i,j)\in S$ and $i>1$, also $(i-1,\rho_i(j))\in S$.
For each $i$, let $S_i:=\{j : (i,j)\in S\}$ be the set of positions
of the cells of row $i$ that lie in $S$, and let
$\sigma_i:S_i\to\{1,\ldots,|S_i|\}$ be the unique order-preserving
bijection. Let $T'=(R'_1,\ldots,R'_m)$ be the row array obtained by
keeping, in each row, exactly the entries whose cell lies in $S$:
\[
R'_{i,j}:=R_{i,\sigma_i^{-1}(j)}, \qquad 1\le j\le|S_i|.
\]
For $i>1$, define
\[
\rho'_i:=\sigma_{i-1}\circ\rho_i\circ\sigma_i^{-1} :
\{1,\ldots,|R'_i|\}\to\{1,\ldots,|R'_{i-1}|\},
\]
which is well defined since $\rho_i(S_i)\subseteq S_{i-1}$, as $S$ is
downward closed. We call $T'$ the \emph{restriction} of $T$ to $S$.
\end{definition}

In words: $T'$ is obtained by erasing the cells outside $S$, and
$\rho'$ lets every surviving entry sit above the same entry as before;
the maps $\sigma_i$ only renumber the surviving positions; see
Figure~\ref{fig:restriction}.

\begin{figure}[H]
\centering
\begin{minipage}{0.45\textwidth}
\centering
\begin{tikzpicture}[scale=0.8, every node/.style={font=\small}]
\foreach \i/\v in {1/1,2/2,3/4,4/6} {
  \draw (\i-1,0) rectangle (\i,1);
  \node at (\i-0.5,0.5) {$\v$};
}
\foreach \i/\v/\c in {2/3/gray!30,3/5/white,4/7/white} {
  \draw[fill=\c] (\i-1,1) rectangle (\i,2);
  \node at (\i-0.5,1.5) {$\v$};
}
\foreach \i/\v/\c in {3/6/white,4/8/gray!30} {
  \draw[fill=\c] (\i-1,2) rectangle (\i,3);
  \node at (\i-0.5,2.5) {$\v$};
}
\node[anchor=east] at (-0.3,0.5) {$R_1$};
\node[anchor=east] at (-0.3,1.5) {$R_2$};
\node[anchor=east] at (-0.3,2.5) {$R_3$};
\end{tikzpicture}

\smallskip
$(T,\rho)$, with the cells outside $S$ shaded.
\end{minipage}
\hfill
\begin{minipage}{0.45\textwidth}
\centering
\begin{tikzpicture}[scale=0.8, every node/.style={font=\small}]
\foreach \i/\v in {1/1,2/2,3/4,4/6} {
  \draw (\i-1,0) rectangle (\i,1);
  \node at (\i-0.5,0.5) {$\v$};
}
\foreach \i/\v in {3/5,4/7} {
  \draw (\i-1,1) rectangle (\i,2);
  \node at (\i-0.5,1.5) {$\v$};
}
\foreach \i/\v in {3/6} {
  \draw (\i-1,2) rectangle (\i,3);
  \node at (\i-0.5,2.5) {$\v$};
}
\node[anchor=east] at (-0.3,0.5) {$R'_1$};
\node[anchor=east] at (-0.3,1.5) {$R'_2$};
\node[anchor=east] at (-0.3,2.5) {$R'_3$};
\end{tikzpicture}

\smallskip
$(T',\rho')$.
\end{minipage}
\caption{The row array $R_1=(1,2,4,6)$, $R_2=(3,5,7)$, $R_3=(6,8)$
with a positioning $\rho$, and its restriction to the set $S$ of
unshaded cells. The entry $5$ moves from position $2$ to position $1$
of its row ($\sigma_2(2)=1$), but still sits above $4$. $T'$ is drawn
with $\rho'$, which is a positioning by
Lemma~\ref{lem:positioning-restricts}.}
\label{fig:restriction}
\end{figure}
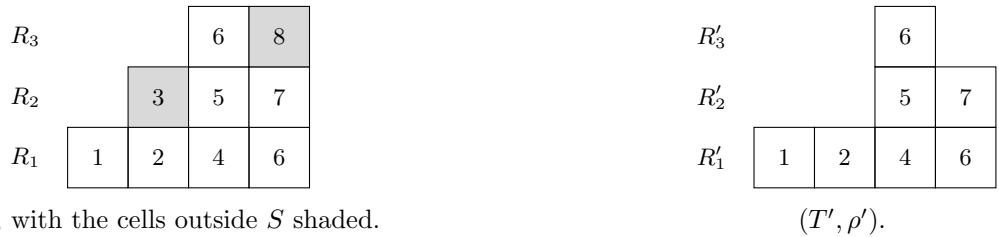

\begin{lemma}[Restricting a positioning]
\label{lem:positioning-restricts}
In the notation of Definition~\ref{def:restriction}, $\rho'$ is a
positioning on $T'$. We call the aligned row array $(T',\rho')$ the
\emph{restriction} of $(T,\rho)$ to $S$.
\end{lemma}

\begin{proof}
Fix $i>1$ and $1\le j\le|R'_i|$.

(P1): by (P1) for $\rho$,
\[
R'_{i-1,\,\rho'_i(j)}
=R'_{i-1,\,\sigma_{i-1}(\rho_i(\sigma_i^{-1}(j)))}
=R_{i-1,\,\rho_i(\sigma_i^{-1}(j))}<R_{i,\,\sigma_i^{-1}(j)}=R'_{i,j}.
\]

(P2): $\rho'_i$ is a composition of the strictly increasing maps
$\sigma_i^{-1}$, $\rho_i$ (by (P2) for $\rho$) and $\sigma_{i-1}$,
hence strictly increasing.
\end{proof}

\begin{lemma}[Restricting a tight positioning stays tight]
\label{lem:tight-restricts}
In the notation of Definition~\ref{def:restriction}, suppose that
$\rho=\rho^{\mathrm{tight}}$ is the tight positioning on $T$. Then
$\rho'$ is the tight positioning on $T'$.
\end{lemma}

\begin{proof}
Fix $i>1$ and a cell $(i,j)\in S$, with entry $x:=R_{i,j}$. In
$(T,\rho^{\mathrm{tight}})$, $x$ sits above $y:=R_{i-1,q}$, where
$q=\rho^{\mathrm{tight}}_i(j)$; so $y$ is the rightmost entry of
$R_{i-1}$ smaller than $x$, and every entry to its right in $R_{i-1}$
is $\ge x$. As $S$ is downward closed, $y$ survives in $R'_{i-1}$, and
by definition of $\rho'$, $x$ still sits above $y$ in $(T',\rho')$.
The entries to the right of $y$ in $R'_{i-1}$ are among those to its
right in $R_{i-1}$, hence $\ge x$. So $y$ is still the rightmost entry
of $R'_{i-1}$ smaller than $x$, i.e.\ $\rho'_i(\sigma_i(j))$ is the
tight value computed on $T'$.
\end{proof}

\begin{lemma}[Keeping row $1$ keeps every column]
\label{lem:restriction-keeps-columns}
In the notation of Definition~\ref{def:restriction}, write $C$
and $C'$ for the column maps of $(T,\rho)$ and $(T',\rho')$
(Definition~\ref{def:column} and
Lemma~\ref{lem:positioning-restricts}). If $S$ contains every cell of row $1$,
then $R'_1=R_1$, and
\[
C'(i,\sigma_i(j))=C(i,j) \qquad \text{for every }(i,j)\in S:
\]
every surviving entry lies in the same column of $(T',\rho')$ as it
did in $(T,\rho)$.
\end{lemma}

\begin{proof}
Since $S_1=\{1,\ldots,|R_1|\}$, the map $\sigma_1$ is an
order-preserving bijection of $\{1,\ldots,|R_1|\}$ onto itself, hence
the identity; so $R'_1=R_1$.

We prove $C'(i,\sigma_i(j))=C(i,j)$ by induction on $i$. For $i=1$,
$C'(1,\sigma_1(j))=\sigma_1(j)=j=C(1,j)$. For $i>1$,
\begin{align*}
C'(i,\sigma_i(j)) &= C'(i-1,\rho'_i(\sigma_i(j)))
  && \text{(Definition~\ref{def:column}, in $(T',\rho')$)}\\
&= C'(i-1,\sigma_{i-1}(\rho_i(j)))
  && \text{(definition of $\rho'$)}\\
&= C(i-1,\rho_i(j))
  && \text{(induction, as $(i-1,\rho_i(j))\in S$)}\\
&= C(i,j)
  && \text{(Definition~\ref{def:column}, in $(T,\rho)$)}. \qedhere
\end{align*}
\end{proof}

\subsection{Full columns and the full core}

\begin{definition}[Full column]
\label{def:full-column}
A column $C_r$ (Definition~\ref{def:column}), reaching rows
$1,\ldots,|C_r|$ (Definition~\ref{def:column-value}), is \emph{full}
if
\[
C_r(i) = C_r(1)+(i-1) \qquad \text{for every } i=1,\ldots,|C_r|
\]
(equivalently: its word has no gaps).
\end{definition}

\begin{definition}[Full aligned row array]
An aligned row array $(T,\rho)$ (Definition~\ref{def:aligned-row-array})
is \emph{full} if every column of $(T,\rho)$ (Definition~\ref{def:column})
is full (see Figure~\ref{fig:full-vs-nonfull}).
\end{definition}

\begin{remark}
The definition makes sense for any aligned row array, but we use this
notion only for an $N$-tableau $T$ equipped with its tight positioning $\rho=\rho^{\mathrm{tight}}$; if $(T,\rho)$ is full,
we call $T$ a \emph{full $N$-tableau}.
\end{remark}

\begin{figure}[H]
\centering
\begin{minipage}{0.45\textwidth}
\centering
\begin{tikzpicture}[scale=0.8, every node/.style={font=\small}]
\foreach \i/\v in {1/1,2/2,3/3,4/4} {
  \draw (\i-1,0) rectangle (\i,1);
  \node at (\i-0.5,0.5) {$\v$};
}
\foreach \i/\v in {1/2,2/3,3/4} {
  \draw (\i-1,1) rectangle (\i,2);
  \node at (\i-0.5,1.5) {$\v$};
}
\foreach \i/\v in {2/4} {
  \draw (\i-1,2) rectangle (\i,3);
  \node at (\i-0.5,2.5) {$\v$};
}
\node[anchor=east] at (-0.3,0.5) {$R_1$};
\node[anchor=east] at (-0.3,1.5) {$R_2$};
\node[anchor=east] at (-0.3,2.5) {$R_3$};
\end{tikzpicture}

\smallskip
Full: every column consecutive.
\end{minipage}
\hfill
\begin{minipage}{0.45\textwidth}
\centering
\begin{tikzpicture}[scale=0.8, every node/.style={font=\small}]
\foreach \i/\v in {1/1,2/2,3/4,4/5} {
  \draw (\i-1,0) rectangle (\i,1);
  \node at (\i-0.5,0.5) {$\v$};
}
\foreach \i/\v in {1/2,3/5} {
  \draw (\i-1,1) rectangle (\i,2);
  \node at (\i-0.5,1.5) {$\v$};
}
\foreach \i/\v in {1/5} {
  \draw (\i-1,2) rectangle (\i,3);
  \node at (\i-0.5,2.5) {$\v$};
}
\node[anchor=east] at (-0.3,0.5) {$R_1$};
\node[anchor=east] at (-0.3,1.5) {$R_2$};
\node[anchor=east] at (-0.3,2.5) {$R_3$};
\end{tikzpicture}

\smallskip
Not full: $C_1$ jumps from $2$ to $5$.
\end{minipage}
\caption{Two tight positioned $N$-tableaux. The left is full; the
right is not, since $C_1$ is not a full column.}
\label{fig:full-vs-nonfull}
\end{figure}

\begin{definition}[Gap in a column]
\label{def:column-gap}
A column $C_r$, reaching rows $1,\ldots,|C_r|$
(Definition~\ref{def:column-value}), has a \emph{gap} at row $i$
(for $1<i\le|C_r|$) if $C_r(i)\ne C_r(i-1)+1$. By
Definition~\ref{def:full-column}, $C_r$ is full iff it has no gap.

Write $g(C_r)$ for the row just below the first gap of $C_r$ when
climbing from row $1$ upward: the largest $i\in\{1,\ldots,|C_r|\}$
such that $C_r$ has no gap at any row $\le i$ (so $g(C_r)=|C_r|$
exactly when $C_r$ is full); we write $g_r:=g(C_r)$ as a shorthand.
\end{definition}

\begin{definition}[Full core of $T$]
\label{def:full-core}
Let $(T,\rho^{\mathrm{tight}})$ be a tight positioned $N$-tableau
(Corollary~\ref{cor:tight-exists-for-Ntableau})
with columns $C_1,\ldots,C_k$, $C_r$ reaching rows
$1,\ldots,|C_r|$, and let $g_r$ be as in
Definition~\ref{def:column-gap}. Let
\[
S^f:=\bigl\{(i,j) : i\le g_{C(i,j)}\bigr\}
\]
be the set of cells lying, in their column $C_r$, no higher than row
$g_r$. It is downward closed along columns: if $(i,j)\in S^f$ and
$i>1$, then $(i-1,\rho_i(j))$ lies in the same column
(Definition~\ref{def:column}) and in a lower row. Let
$T^f=(R_1^f,\ldots,R_m^f)$, the \emph{full core of $T$}, be the restriction of $T$ to $S^f$, in the sense of
Definition~\ref{def:restriction}: so $R_i^f$ keeps exactly the
entries of $R_i$ lying in columns $C_r$ with $i\le g_r$. Since $g_r\ge1$ for
every $r$, $S^f$ contains every cell of row $1$; so, by
Lemma~\ref{lem:restriction-keeps-columns}, $R_1^f=R_1$, and every
entry of $T^f$ lies in the same column as it did in $T$. Since
$g_r=|C_r|$ exactly when $C_r$ is full, $(T,\rho^{\mathrm{tight}})$ is
full if and only if $T^f=T$.
\end{definition}

Since the tight positioning of an $N$-tableau $T$ is uniquely
determined by $T$, the full core $T^f$ also depends only on $T$. We
never consider the analogous construction for other positionings.

\begin{remark}
\label{rem:full-core-tight}
Let $(T,\rho)$ be a tight positioned $N$-tableau and
$T^f$ its full core. By Lemma~\ref{lem:tight-restricts},
$\rho=\rho^{\mathrm{tight}}$ induces the tight positioning $\rho^f$ on
$T^f$ -- we denote it also by $\rho$ for simplicity.
\end{remark}

\begin{example}[A full core]
\label{ex:full-core}
Let $R_1=(1,2,4,5,7)$, $R_2=(2,4,5,7)$, $R_3=(5,7)$, $R_4=(7)$. This
is an $N$-tableau $T$; let $\rho$ be its tight positioning.
Figure~\ref{fig:full-core} shows $(T,\rho)$ and $(T^f,\rho)$. Here $C_2$ has two gaps, at rows $2$ and
$4$, $C_3$ has a gap at row $3$, and $C_4$ has a gap at row $2$, so
$g_2=g_4=1$ and $g_3=2$. This leaves $R_1^f=(1,2,4,5,7)$ and
$R_2^f=(2,5)$, with all columns full.
\end{example}

\begin{figure}[H]
\centering
\begin{minipage}{0.45\textwidth}
\centering
\begin{tikzpicture}[scale=0.8, every node/.style={font=\small}]
\foreach \i/\v in {1/1,2/2,3/4,4/5,5/7} {
  \draw (\i-1,0) rectangle (\i,1);
  \node at (\i-0.5,0.5) {$\v$};
}
\foreach \i/\v/\c in {1/2/white,2/4/gray!30,3/5/white,4/7/gray!30} {
  \draw[fill=\c] (\i-1,1) rectangle (\i,2);
  \node at (\i-0.5,1.5) {$\v$};
}
\foreach \i/\v in {2/5,3/7} {
  \draw[fill=gray!30] (\i-1,2) rectangle (\i,3);
  \node at (\i-0.5,2.5) {$\v$};
}
\foreach \i/\v in {2/7} {
  \draw[fill=gray!30] (\i-1,3) rectangle (\i,4);
  \node at (\i-0.5,3.5) {$\v$};
}
\node[anchor=east] at (-0.3,0.5) {$R_1$};
\node[anchor=east] at (-0.3,1.5) {$R_2$};
\node[anchor=east] at (-0.3,2.5) {$R_3$};
\node[anchor=east] at (-0.3,3.5) {$R_4$};
\end{tikzpicture}

\smallskip
$(T,\rho)$, with the deleted cells shaded.
\end{minipage}
\hfill
\begin{minipage}{0.45\textwidth}
\centering
\begin{tikzpicture}[scale=0.8, every node/.style={font=\small}]
\foreach \i/\v in {1/1,2/2,3/4,4/5,5/7} {
  \draw (\i-1,0) rectangle (\i,1);
  \node at (\i-0.5,0.5) {$\v$};
}
\foreach \i/\v in {1/2,3/5} {
  \draw (\i-1,1) rectangle (\i,2);
  \node at (\i-0.5,1.5) {$\v$};
}
\node[anchor=east] at (-0.3,0.5) {$R_1^f$};
\node[anchor=east] at (-0.3,1.5) {$R_2^f$};
\end{tikzpicture}

\smallskip
$(T^f,\rho)$.
\end{minipage}
\caption{An $N$-tableau with its tight positioning, and its full core
(Example~\ref{ex:full-core}).}
\label{fig:full-core}
\end{figure}

As illustrated in Example~\ref{ex:full-core}, some rows of $T^f$ may
turn out to be empty; the next remark shows why these may simply be
dropped.

\begin{remark}[Empty rows of $T^f$ form a top suffix]
\label{rem:full-core-empty-rows}
By Remark~\ref{rem:full-core-tight}, $\rho$ is a positioning on
$T^f$, so by Lemma~\ref{lem:sizes-decrease} its row sizes are
non-increasing: if $R_i^f=\varnothing$ for some $i$, then
$R_j^f=\varnothing$ for every $j>i$ too. So the empty rows of $T^f$,
if any, form a suffix at the top, and we may simply drop them: from
now on, $T^f$ denotes the row array obtained after deleting this
(possibly empty) top suffix, and $\rho$ the positioning it inherits.
\end{remark}

\subsection{The full core is an \texorpdfstring{$N$}{N}-tableau}

By construction $(T^f,\rho)$ is a full row array. In this subsection
we show that it is also an $N$-tableau. The row-minima condition
follows from Remark~\ref{rem:positioning-gives-minima}, so the main
point is the containment condition.

\begin{proposition}[The full core preserves containment]
\label{prop:full-core-preserves-containment}
Let $(T,\rho^{\mathrm{tight}})$ be a tight positioned $N$-tableau and
$T^f$ its full core (Definition~\ref{def:full-core}). Then
$R_i^f\subseteq R_{i-1}^f$ for every $i>1$.
\end{proposition}

\begin{proof}
Fix $i>1$ and $a\in R_i^f$, say $a=C_r(i)$ for a column with $i\le
g_r$, i.e.\ $C_r$ has no gap through row $i$. Write $a_t:=C_r(t)$
for $t=1,\ldots,i$, so $a_i=a$ and $a_t=a_{t-1}+1$.

We show by induction on $t=2,\ldots,i$ that $a_t\in R_{t-1}^f$.

\emph{Base case $t=2$}: $a_2\in R_2\subseteq R_1$ ($N$-tableau
containment), and $R_1^f=R_1$ (Definition~\ref{def:full-core}), so
$a_2\in R_1^f$.

\emph{Inductive step}: assume the claim holds for $t-1$, i.e.\
$a_{t-1}\in R_{t-2}^f$: say $a_{t-1}$ occurs at position $q_1$ in
row $t-2$, in a column $C_{r'}$ with $t-2\le g_{r'}$. Since
$a_t\in R_t\subseteq R_{t-1}$, it occurs at some position $q_2$ in
row $t-1$. As $R_{t-2,q_1}+1=a_{t-1}+1=a_t=R_{t-1,q_2}$,
Lemma~\ref{lem:successor-above} (applied to row $t-2$) gives
$\rho_{t-1}(q_2)=q_1$: $(t-1,q_2)$ and $(t-2,q_1)$ lie in the same
column $C_{r'}$, so $C_{r'}(t-1)=a_t$. Since $C_{r'}$ has no gap
through row $t-2$, and $C_{r'}(t-1)=C_{r'}(t-2)+1$, it has none
through row $t-1$ either: $t-1\le g_{r'}$. So $a_t\in R_{t-1}^f$.

Taking $t=i$ gives $a=a_i\in R_{i-1}^f$.
\end{proof}

\begin{corollary}[The full core of an $N$-tableau is an $N$-tableau]
\label{cor:full-core-is-Ntableau}
Let $(T,\rho^{\mathrm{tight}})$ be a tight positioned $N$-tableau and
$T^f$ its full core (Definition~\ref{def:full-core}). Then $T^f$ is
an $N$-tableau.
\end{corollary}

\begin{proof}
The rows of $T^f$ are subsequences of those of $T$, hence strictly
increasing. They are nonempty (Remark~\ref{rem:full-core-empty-rows})
and $\rho$ is a positioning on $T^f$
(Remark~\ref{rem:full-core-tight}), so the row-minima condition holds
by Remark~\ref{rem:positioning-gives-minima}. Containment is
Proposition~\ref{prop:full-core-preserves-containment}.
\end{proof}

\section{Reducing to the full core}
\label{sec:catalan-image}
Recall from Section~\ref{sec:preliminaries} that $\mathbf{n}=\{1,\ldots,n\}$.

\begin{definition}[The Catalan monoid]
\label{def:catalan-monoid}
We write maps on the right: for
$\alpha:\{0,1,\ldots,n\}\to\{0,1,\ldots,n\}$ and $x\in\{0,1,\ldots,n\}$,
$x\alpha$ denotes the value of $\alpha$ at $x$. A map $\alpha$ is
\emph{order-preserving} if $x\le y$ implies $x\alpha\le y\alpha$, and
\emph{order-decreasing} if $x\alpha\le x$ for every $x\in\{0,1,\ldots,n\}$.
The \emph{Catalan monoid} $\mathrm{Cat}_n$ is the set of all
order-preserving, order-decreasing self-maps of $\{0,1,\ldots,n\}$,
under the product $\alpha\beta$ defined by $x(\alpha\beta):=(x\alpha)\beta$ -- composition read
left to right, first $\alpha$ then $\beta$.
\end{definition}

\begin{definition}[$\epsilon$-maps]
\label{def:delta-map}
For $x\in\{1,\ldots,n\}$, define $\epsilon_x:\{0,1,\ldots,n\}\to
\{0,1,\ldots,n\}$ by $x\epsilon_x:=x-1$ and $y\epsilon_x:=y$ for
every $y\ne x$. Each $\epsilon_x$ lies in $\mathrm{Cat}_n$.
\end{definition}

\begin{theorem}[Volkov \cite{Volkov}]
\label{thm:volkov-homomorphism}
The assignment $x\mapsto\epsilon_x$, for $x\in\{1,\ldots,n\}$, extends
to a surjective monoid homomorphism
\[
\Phi:\mathrm{Styl}_n\twoheadrightarrow\mathrm{Cat}_n.
\]
\end{theorem}

Given an $N$-tableau $T$,
$\Phi(T)$ is concretely computable: take any word representing $T$'s
class in $\mathrm{Styl}_n$ -- its row word or column word
(Section~\ref{sec:preliminaries}) both work -- and compose the maps
$\epsilon_x$ (Definition~\ref{def:delta-map}) over that word's
letters $x$, in order. Our goal is instead a visual, graphical
description of $\Phi$: to be able to look at the picture of an
$N$-tableau $T$ and read off directly which element of
$\mathrm{Cat}_n$ it is, without going through reading a row or column
word.

Let $T$ be an $N$-tableau and $T^f$ its full core
(Definition~\ref{def:full-core}). Our next step, to which the
remainder of this section is devoted, is to show that
$\Phi(T)=\Phi(T^f)$ (Theorem~\ref{thm:phi-full-core}): passing to the
full core leaves the image under $\Phi$ unchanged. To prove it we
compute both sides through column words. By
Corollary~\ref{cor:col-well-defined}, the column word of any
positioning of an $N$-tableau represents it in $\mathrm{Styl}_n$. We
use the tight positioning $\rho$ of $T$, on which the full core is
built, and the tight positioning that $T^f$ inherits from it, also
denoted $\rho$ (Remark~\ref{rem:full-core-tight}), and show that
$\Phi(\mathrm{col}(T,\rho))=\Phi(\mathrm{col}(T^f,\rho))$.

For elements $a\le b$ of $\mathbf{n}$, we write
$v_{b,a}:=b\,(b-1)\cdots a$ for the word listing $a,\ldots,b$ in
decreasing order, and
$\epsilon_{b,a}:=\Phi(v_{b,a})=\epsilon_b\,\epsilon_{b-1}\cdots\epsilon_a$
for its image in $\mathrm{Cat}_n$. These abbreviations will be
convenient in the statements below.

\begin{lemma}[A descending run collapses to one less than its base]
\label{lem:run-collapses-interval}
Let $a\le b$ be elements of $\mathbf{n}$. Then $y\epsilon_{b,a}=a-1$ for every $y\in\{a,\ldots,b\}$, and
$y\epsilon_{b,a}=y$ for every $y\in\{0,\ldots,n\}\setminus\{a,\ldots,b\}$.
\end{lemma}

\begin{proof}
If $y\notin\{a,\ldots,b\}$, no letter of $v_{b,a}$ equals $y$, so every
$\epsilon_c$ in $\epsilon_{b,a}$ fixes $y$, giving $y\epsilon_{b,a}=y$.

Now let $y\in\{a,\ldots,b\}$. The maps $\epsilon_b,\ldots,\epsilon_{y+1}$,
applied first, fix $y$; then $\epsilon_y,\epsilon_{y-1},\ldots,\epsilon_a$
send it to $y-1,y-2,\ldots,a-1$ in turn. So $y\epsilon_{b,a}=a-1$.
\end{proof}

\begin{lemma}[A word avoiding $a-1$ preserves $\{a-1,\ldots,b\}$]
\label{lem:absorbing-interval}
Let $a\le b$ be elements of $\mathbf{n}$, and let $w:=x_1x_2\cdots x_k$
be a word over $\mathbf{n}$ with $x_i\ne a-1$ for every $i$. Write
$\alpha:=\Phi(w)=\epsilon_{x_1}\epsilon_{x_2}\cdots\epsilon_{x_k}$.
Then $y\alpha\in\{a-1,\ldots,b\}$ for every $y\in\{a-1,\ldots,b\}$.
\end{lemma}

\begin{proof}
Induction on $k$; trivial for $k=0$ ($\alpha$ is the identity).
For $k\ge1$, write $\alpha':=\Phi(x_1\cdots x_{k-1})$, so
$\alpha=\alpha'\epsilon_{x_k}$, and let
$y':=y\alpha'\in\{a-1,\ldots,b\}$ by the induction hypothesis. If
$x_k\ne y'$, then $y'\epsilon_{x_k}=y'\in\{a-1,\ldots,b\}$. If
$x_k=y'$: since $x_k\ne a-1$ and $y'\in\{a-1,\ldots,b\}$, in
fact $y'\in\{a,\ldots,b\}$, so $y'\epsilon_{x_k}=y'-1\in
\{a-1,\ldots,b-1\}\subseteq\{a-1,\ldots,b\}$. Either way
$y\alpha=y'\epsilon_{x_k}\in\{a-1,\ldots,b\}$.
\end{proof}

\begin{lemma}[Deleting a repeated descending run]
\label{lem:delete-repeated-run}
Let $a\le b$ be elements of $\mathbf{n}$, and let $w$ be a word over
$\mathbf{n}$ that does not contain the letter $a-1$. Then
\[
\Phi(v_{b,a}\,w\,v_{b,a})=\Phi(w\,v_{b,a}).
\]
\end{lemma}

\begin{proof}
Write $\beta:=\epsilon_{b,a}=\Phi(v_{b,a})$ and $\alpha:=\Phi(w)$. Since $\Phi$ is a
homomorphism (Theorem~\ref{thm:volkov-homomorphism}),
$\Phi(v_{b,a}\,w\,v_{b,a})=\beta\alpha\beta$ and
$\Phi(w\,v_{b,a})=\alpha\beta$,
so we must show $y\beta\alpha\beta=y\alpha\beta$ for every
$y\in\{0,\ldots,n\}$. By
Lemma~\ref{lem:run-collapses-interval}, $\beta$ sends every element of
$\{a,\ldots,b\}$ to $a-1$ and fixes everything else.

If $y\notin\{a,\ldots,b\}$, then $y\beta=y$, so both sides equal
$y\alpha\beta$. If $y\in\{a,\ldots,b\}$, then $y\beta=a-1$, so
$y\beta\alpha=(a-1)\alpha$. Since both $y$ and $a-1$ lie in
$\{a-1,\ldots,b\}$, Lemma~\ref{lem:absorbing-interval} gives
$y\alpha,\,y\beta\alpha\in\{a-1,\ldots,b\}$. By
Lemma~\ref{lem:run-collapses-interval}, $\beta$ sends every element of
this set to $a-1$ (it fixes $a-1$ and sends $a,\ldots,b$ to $a-1$).
Hence $y\beta\alpha\beta=a-1=y\alpha\beta$.
\end{proof}

\begin{definition}[Removable factor]
\label{def:removable}
Let $W=u_1\,B\,u_2$ be a word over $\mathbf{n}$, where $B=v_{b,a}$ for
some $a\le b$. The factor $B$ is \emph{removable} in $W$ if it is
followed later in $W$ by another factor $v_{b,a}$, with no letter $a-1$
in between; that is, if $u_2=w\,v_{b,a}\,u_3$ with $w$ not containing
$a-1$.
\end{definition}

Whether $B$ is removable depends only on the suffix $B\,u_2$: if $B$ is
removable in $u_1\,B\,u_2$, it is removable in $u\,B\,u_2$ for every
word $u$. By Lemma~\ref{lem:delete-repeated-run}, since $\Phi$ is a
homomorphism, deleting a removable factor does not change $\Phi$:
$\Phi(u_1\,v_{b,a}\,w\,v_{b,a}\,u_3)=\Phi(u_1\,w\,v_{b,a}\,u_3)$.

\begin{corollary}[Deleting removable factors]
\label{cor:delete-several-runs}
Let $W=u_0\,B_1\,u_1\,B_2\,u_2\cdots B_N\,u_N$ be a word over
$\mathbf{n}$ in which each $B_s$ is a removable factor. Then
$\Phi(W)=\Phi(u_0\,u_1\cdots u_N)$.
\end{corollary}

\begin{proof}
Induction on $N$; for $N=0$ there is nothing to prove. Deleting $B_1$
gives $W_1:=(u_0u_1)\,B_2\,u_2\cdots B_N\,u_N$, and
$\Phi(W_1)=\Phi(W)$, since $B_1$ is removable. For $s\ge2$, the suffix of
$W_1$ starting at $B_s$ is the same as in $W$, so $B_s$ is still
removable in $W_1$. By induction, $\Phi(W_1)=\Phi(u_0\,u_1\cdots u_N)$.
\end{proof}

We now return to $N$-tableaux, and show below
(Proposition~\ref{prop:deletion-preserves-phi}) that every run above a
gap in a column forms a removable factor of the column word.

\begin{lemma}[A copy of a run above a gap]
\label{lem:run-copies}
Let $T=(R_1,\ldots,R_m)$ be an $N$-tableau equipped with its tight
positioning $\rho=\rho^{\mathrm{tight}}$
(Corollary~\ref{cor:tight-exists-for-Ntableau}). Let $C_r$ be a column,
$i>1$, and $a\le b$ with $k:=b-a$, such that $C_r(i+\ell)=a+\ell$ for
$\ell=0,\ldots,k$ and $C_r(i-1)\ne a-1$. Then (see Figure~\ref{fig:run-copy}):
\begin{enumerate}
\item[(a)] the run is repeated one row lower, in a column to the right
of $C_r$: there is a column $C_{r'}$ with $r<r'$ such that
$C_{r'}(i-1+\ell)=a+\ell$ for $\ell=0,\ldots,k$ -- we call these cells
the \emph{copy} of the run;
\item[(b)] every occurrence of the letter $a-1$ is either directly
below the copy or in a column further right: it lies in a column
$C_{r''}$ with $r''\ge r'$, and if $r''=r'$, then $C_{r'}(i-2)=a-1$.
\end{enumerate}
\end{lemma}

\begin{figure}[H]
\centering
\begin{tikzpicture}[scale=0.7, every node/.style={font=\small}]
\draw (0,0) rectangle (1.4,1); \node at (0.7,0.5) {$\ast$};
\draw (0,1) rectangle (1.4,2); \node at (0.7,1.5) {$\ast$};
\draw (0,2) rectangle (1.4,3); \node at (0.7,2.5) {$\ast$};
\draw (0,3) rectangle (1.4,4); \node at (0.7,3.5) {$\ast$};
\draw (0,4) rectangle (1.4,5); \node at (0.7,4.55) {$\vdots$};
\draw (0,5) rectangle (1.4,6); \node at (0.7,5.5) {$\ast$};
\draw (1.4,0) rectangle (2.8,1); \node at (2.1,0.5) {$\ast$};
\draw (1.4,1) rectangle (2.8,2); \node at (2.1,1.5) {$\ast$};
\draw (1.4,2) rectangle (2.8,3); \node at (2.1,2.5) {$c$};
\draw[fill=gray!30] (1.4,3) rectangle (2.8,4); \node at (2.1,3.5) {$a$};
\draw[fill=gray!30] (1.4,4) rectangle (2.8,5); \node at (2.1,4.55) {$\vdots$};
\draw[fill=gray!30] (1.4,5) rectangle (2.8,6); \node at (2.1,5.5) {$b$};
\draw (2.8,0) rectangle (4.2,1); \node at (3.5,0.5) {$\ast$};
\draw (2.8,1) rectangle (4.2,2); \node at (3.5,1.5) {$\ast$};
\draw (2.8,2) rectangle (4.2,3); \node at (3.5,2.5) {$\ast$};
\draw (4.2,0) rectangle (5.6,1); \node at (4.9,0.5) {$\ast$};
\draw (4.2,1) rectangle (5.6,2); \node at (4.9,1.5) {$a-1$};
\draw[fill=gray!30] (4.2,2) rectangle (5.6,3); \node at (4.9,2.5) {$a$};
\draw[fill=gray!30] (4.2,3) rectangle (5.6,4); \node at (4.9,3.55) {$\vdots$};
\draw[fill=gray!30] (4.2,4) rectangle (5.6,5); \node at (4.9,4.5) {$b$};
\draw (5.6,0) rectangle (7,1); \node at (6.3,0.5) {$a-1$};
\draw (5.6,1) rectangle (7,2); \node at (6.3,1.5) {$\ast$};
\draw (7,0) rectangle (8.4,1); \node at (7.7,0.5) {$\ast$};
\node at (2.1,-0.5) {$C_r$};
\node at (4.9,-0.5) {$C_{r'}$};
\node at (6.3,-0.5) {$C_{r''}$};
\end{tikzpicture}
\caption{Lemma~\ref{lem:run-copies}, schematically. The run (shaded,
in $C_r$) sits above an entry $c\ne a-1$; its copy (shaded, in
$C_{r'}$) is one row lower. An occurrence of $a-1$ is either directly
below the copy or in a column $C_{r''}$ further right.}
\label{fig:run-copy}
\end{figure}

\begin{proof}
(a) For $\ell=0,\ldots,k$, containment gives
$a+\ell\in R_{i+\ell}\subseteq R_{i-1+\ell}$; as rows are strictly
increasing, $a+\ell$ occurs in $R_{i-1+\ell}$ at a unique position
$q_\ell$. For $\ell\ge1$, since
$R_{i-2+\ell,q_{\ell-1}}+1=a+\ell=R_{i-1+\ell,q_\ell}$,
Lemma~\ref{lem:successor-above} puts $(i-1+\ell,q_\ell)$ directly above
$(i-2+\ell,q_{\ell-1})$, in the same column. So these cells lie in a
single column $C_{r'}$, i.e.\ $C_{r'}(i-1+\ell)=a+\ell$ for
$\ell=0,\ldots,k$. Since $C_r(i)=a=C_{r'}(i-1)$,
Lemma~\ref{lem:equal-entry-right} gives $r<r'$.

(b) First, $a-1\notin R_{i-1}$: otherwise $a-1$ would be the largest
entry of $R_{i-1}$ smaller than $a$, so the tight positioning would
place the entry $a$ of row $i$ directly above it, i.e.\
$C_r(i-1)=a-1$, contrary to hypothesis. By containment, $a-1$ then
occurs in no row $h\ge i-1$. So let $a-1$ occur at a cell $(t-1,q)$
with $t\le i-1$, in a column $C_{r''}$. Since $R_{t-1,q}+1=a$,
Lemma~\ref{lem:successor-above} puts the $a$ of row $t$ directly above
$(t-1,q)$, in $C_{r''}$. Comparing it with the $a$ of row $i-1$, which
lies in $C_{r'}$, Corollary~\ref{cor:higher-left} gives $r''>r'$
if $t<i-1$, and $r''=r'$ if $t=i-1$. In the latter case
$(t-1,q)=(i-2,q)$ lies in $C_{r'}$, i.e.\ $C_{r'}(i-2)=a-1$.
\end{proof}

\begin{proposition}[A run above a gap and its copy in the column word]
\label{prop:deletion-preserves-phi}
Let $T=(R_1,\ldots,R_m)$ be an $N$-tableau equipped with its tight
positioning $\rho=\rho^{\mathrm{tight}}$, and let $C_r$, $i$, $a\le b$
be as in Lemma~\ref{lem:run-copies}. Then the letters of the run in
$C_r$ form a removable factor of $\mathrm{col}(T,\rho)$
(Definition~\ref{def:removable}), the later factor being formed by the
letters of its copy in $C_{r'}$ (Lemma~\ref{lem:run-copies}(a)).
\end{proposition}

\begin{proof}
The run occupies consecutive rows of $C_r$, and a column is read from
its top row down (Definition~\ref{def:column-word}), so its letters
form a factor $v_{b,a}$ of $\mathrm{col}(T,\rho)$. The copy of
Lemma~\ref{lem:run-copies}(a) occupies consecutive rows of the single
column $C_{r'}$, so it also forms a factor $v_{b,a}$; since $r<r'$, it
comes after the first one.

By Lemma~\ref{lem:run-copies}(b), every occurrence of $a-1$ lies in a
column $C_{r''}$ with $r''\ge r'$; if $r''>r'$ it comes after the whole
column $C_{r'}$, and if $r''=r'$ it is $C_{r'}(i-2)$, which lies
below the copy. Either way it comes after the copy, so there is no
$a-1$ between the two factors.
\end{proof}

We can now prove the main result of this section.

\begin{theorem}[Passing to the full core does not change $\Phi$]
\label{thm:phi-full-core}
Let $(T,\rho)$ be a tight positioned $N$-tableau and $T^f$ its full
core, with its tight positioning, also denoted $\rho$
(Remark~\ref{rem:full-core-tight}). Then
\[
\Phi(\mathrm{col}(T,\rho))=\Phi(\mathrm{col}(T^f,\rho)).
\]
\end{theorem}

\begin{proof}
Let $C_1,\ldots,C_k$ be the columns of $(T,\rho)$. Going up each
column $C_r$ from row $1$, split its cells into maximal runs of
consecutive values $B^r_1,B^r_2,\ldots,B^r_{p_r}$, starting a new run
at every row where $C_r$ has a gap (Definition~\ref{def:column-gap}).
For instance, in Example~\ref{ex:full-core}, the column $C_2$ splits
into the runs $\{2\}$, $\{4,5\}$ and $\{7\}$. By
Definition~\ref{def:full-core}, $B^r_1$ consists of the cells of $C_r$
in rows $1,\ldots,g_r$, i.e.\ it is the column $C_r$ of $T^f$. Each
$B^r_j$ with $j\ge2$ starts at a row $i>1$ where $C_r$ has a gap, so
it satisfies the hypotheses of Lemma~\ref{lem:run-copies}, with $a$ and
$b$ its smallest and largest entries. Since a column is read from top
to bottom and the columns from left to right, reading each run as a
decreasing word we get
\[
\mathrm{col}(T,\rho)=B^1_{p_1}\cdots B^1_2B^1_1\ \cdots\
B^k_{p_k}\cdots B^k_2B^k_1,\qquad
\mathrm{col}(T^f,\rho)=B^1_1B^2_1\cdots B^k_1.
\]

So $\mathrm{col}(T^f,\rho)$ is obtained from $\mathrm{col}(T,\rho)$ by
deleting the runs $B^r_j$ with $j\ge2$. By
Proposition~\ref{prop:deletion-preserves-phi}, each of them is a
removable factor of $\mathrm{col}(T,\rho)$, so
Corollary~\ref{cor:delete-several-runs} gives
$\Phi(\mathrm{col}(T,\rho))=\Phi(\mathrm{col}(T^f,\rho))$.
\end{proof}

\section{The image of a full \texorpdfstring{$N$}{N}-tableau}
\label{sec:full-image}

By Theorem~\ref{thm:phi-full-core}, to describe $\Phi$ it suffices to
describe $\Phi(T)$ for full $N$-tableaux $T$. Throughout this section,
every $N$-tableau is equipped with its tight positioning $\rho$
(Corollary~\ref{cor:tight-exists-for-Ntableau}).

\subsection{The leftmost-column rule and heights}
\label{sec:leftmost-rule}

\begin{lemma}[Columns to the right lie above a base]
\label{lem:right-columns-above-base}
Let $(T,\rho)$ be an aligned row array, $T=(R_1,\ldots,R_m)$, whose
first row $R_1$ is strictly increasing, and let $r<r'$. Then every
entry of $C_{r'}$ is strictly larger than $C_r(1)$. In particular, the
letter $C_r(1)-1$ does not occur in $w(C_{r'})$.
\end{lemma}

\begin{proof}
By (P1), each column strictly increases going up, so every entry of
$C_{r'}$ is at least $C_{r'}(1)=R_{1,r'}$. Since $R_1$ is strictly increasing
and $r<r'$, $R_{1,r'}>R_{1,r}=C_r(1)$.
\end{proof}

\begin{theorem}[$\Phi$ of a full $N$-tableau]
\label{thm:phi-full}
Let $(T,\rho)$ be a full $N$-tableau, let $\ell:=|R_1|$, with columns
$C_1,\ldots,C_\ell$, and write $\alpha:=\Phi(\mathrm{col}(T,\rho))$.
Let $y\in\{0,\ldots,n\}$.
\begin{enumerate}
\item[(a)] If $y$ does not occur in $T$ (in particular if $y=0$), then
$y\alpha=y$.
\item[(b)] Otherwise, let $C_r$ be the leftmost column in which $y$
occurs. Then $y\alpha=C_r(1)-1$.
\end{enumerate}
\end{theorem}

\begin{proof}
For $s=1,\ldots,\ell$ write $\alpha_s:=\Phi(w(C_s))$. Since
$\mathrm{col}(T,\rho)=w(C_1)\cdots w(C_\ell)$
(Definition~\ref{def:column-word}) and $\Phi$ is a homomorphism,
$\alpha=\alpha_1\cdots\alpha_\ell$. By
Definition~\ref{def:delta-map}, if $z$ is not a letter of a word $u$,
then $z\Phi(u)=z$.

(a) $y$ is not a letter of any $w(C_s)$, so $y\alpha_s=y$ for every
$s$, and $y\alpha=y$.

(b) Write $a:=C_r(1)$. For $s<r$, $y$ is not a letter of $w(C_s)$, so
$y\alpha_1\cdots\alpha_{r-1}=y$. Since $C_r$ is full
(Definition~\ref{def:full-column}), $w(C_r)=v_{c,a}$, where
$c$ is the top entry of $C_r$, and $a\le y\le c$; so
$y\alpha_r=y\epsilon_{c,a}=a-1$ by
Lemma~\ref{lem:run-collapses-interval}. For $s>r$, $a-1$ is not a
letter of $w(C_s)$ (Lemma~\ref{lem:right-columns-above-base}), so
$(a-1)\alpha_{r+1}\cdots\alpha_\ell=a-1$. Hence $y\alpha=a-1$.
\end{proof}

For $y\in\{0,\ldots,n\}$, let $h(y)$ be the number of rows of $T$
containing $y$. By containment, these are exactly the rows
$1,\ldots,h(y)$.

\begin{corollary}[$\Phi$ of a full $N$-tableau, by heights]
\label{cor:phi-full-height}
Let $(T,\rho)$ be a full $N$-tableau and $\alpha:=\Phi(\mathrm{col}(T,\rho))$.
Then $y\alpha=y-h(y)$ for every $y\in\{0,\ldots,n\}$.
\end{corollary}

\begin{proof}
If $h(y)=0$, this is Theorem~\ref{thm:phi-full}(a). Otherwise, by
Corollary~\ref{cor:higher-left}, the leftmost column containing $y$ is the
one through its occurrence in row $h(y)$, whose base is $y-h(y)+1$
since the column is full; so $y\alpha=y-h(y)$ by
Theorem~\ref{thm:phi-full}(b).
\end{proof}

\begin{example}
\label{ex:phi-full}
Let $n=7$ and $T=(R_1,R_2)$ with $R_1=(1,2,3,4,5,7)$ and
$R_2=(2,3,5)$. Its tight positioning is drawn in
Figure~\ref{fig:phi-full}; all its columns are full.
\begin{figure}[H]
\centering
\begin{tikzpicture}[scale=0.8, every node/.style={font=\small}]
\foreach \i/\v in {1/1,2/2,3/3,4/4,5/5,6/7} {
  \draw (\i-1,0) rectangle (\i,1);
  \node at (\i-0.5,0.5) {$\v$};
  \node at (\i-0.5,-0.4) {$C_{\i}$};
}
\foreach \i/\v in {1/2,2/3,4/5} {
  \draw (\i-1,1) rectangle (\i,2);
  \node at (\i-0.5,1.5) {$\v$};
}
\node[anchor=east] at (-0.3,0.5) {$R_1$};
\node[anchor=east] at (-0.3,1.5) {$R_2$};
\end{tikzpicture}
\caption{The full $N$-tableau of Example~\ref{ex:phi-full}, in its
tight positioning.}
\label{fig:phi-full}
\end{figure}
\noindent The letters $0$ and $6$ do not occur, so they are fixed.
Both $1$ and $2$ occur first in $C_1$, with base $1$, so they go to
$0$; $3$ occurs first in $C_2$, with base $2$, so it goes to $1$; both
$4$ and $5$ occur first in $C_4$, with base $4$, so they go to $3$; and
$7$ occurs only in $C_6$, so it goes to $6$. Thus
\[
\alpha=\begin{pmatrix}
0&1&2&3&4&5&6&7\\
0&0&0&1&3&3&6&6
\end{pmatrix}.
\]
The same $\alpha$ comes from Corollary~\ref{cor:phi-full-height}:
$h(1)=1$ and $h(2)=2$ give $1,2\mapsto0$; $h(3)=2$ gives $3\mapsto1$;
$h(4)=1$ and $h(5)=2$ give $4,5\mapsto3$; $h(7)=1$ gives $7\mapsto6$;
and the letters with $h(y)=0$ are fixed.
\end{example}

\subsection{The augmented \texorpdfstring{$N$}{N}-tableau}
\label{sec:augmented}

Theorem~\ref{thm:phi-full} treats the letters that do not occur in $T$
separately. Adding one more row at the bottom absorbs them into the
rule of part (b), so that $\Phi$ is read off from the columns of a
single picture; this is the form used in the rest of the section.

\begin{definition}[Augmented $N$-tableau]
\label{def:augmented}
Let $T=(R_1,\ldots,R_m)$ be an $N$-tableau over $\mathbf{n}$, and let
$R_0:=(0,1,\ldots,n)$. The \emph{augmented $N$-tableau} of $T$ is the
row array $T^+:=(R_0,R_1,\ldots,R_m)$ over $\{0,1,\ldots,n\}$, with
$R_0$ as its bottom row. Since $R_1\subseteq R_0$ and $\min R_0=0<\min
R_1$, it is an $N$-tableau over $\{0,1,\ldots,n\}$, so it has a tight
positioning (Corollary~\ref{cor:tight-exists-for-Ntableau}). The
bottom row $R_0$ has exactly one cell with each value $s\in\{0,\ldots,n\}$;
we write $C^+_s$ for the column of $T^+$ whose base is that cell, so
$C^+_s(0)=s$, and $C^+_0,C^+_1,\ldots,C^+_n$ are the columns of $T^+$
from left to right.
\end{definition}

The augmented $N$-tableau is only a picture: it contains the letter
$0$, so it does not represent an element of $\mathrm{Styl}_n$, and we
never apply $\Phi$ to it. If $T$ is full, the tight positioning of
$T^+$ places each $a\in R_1$ directly above $a-1\in R_0$ and leaves the
rows above unchanged. So each column $C_r$ of $T$ just extends one
cell down, to the base $C_r(1)-1$, and every other column of $T^+$ is a
single cell of $R_0$. In particular, $T^+$ is again full.

\begin{corollary}[$\Phi$ of a full $N$-tableau, uniform rule]
\label{cor:phi-full-augmented}
Let $(T,\rho)$ be a full $N$-tableau, $\alpha:=\Phi(\mathrm{col}(T,\rho))$,
and $T^+$ its augmented $N$-tableau. For every $y\in\{0,\ldots,n\}$,
$y\alpha$ is the base of the leftmost column of $T^+$ in which $y$
occurs.
\end{corollary}

\begin{proof}
Immediate from Theorem~\ref{thm:phi-full}: if $y$ does not occur in
$T$, the only column of $T^+$ containing $y$ is the single cell $y$ of
$R_0$; otherwise the leftmost one is the extension of the leftmost
column $C_r$ of $T$ containing $y$, with base $C_r(1)-1$.
\end{proof}

In terms of $T^+$, Corollary~\ref{cor:phi-full-height} reads: $y$ occurs
in rows $0,1,\ldots,h(y)$, and the base of the leftmost column
containing it is $y-h(y)$.

\begin{example}
\label{ex:phi-full-augmented}
The augmented $N$-tableau of the full $N$-tableau of
Example~\ref{ex:phi-full} is drawn in Figure~\ref{fig:augmented}. Reading
off the base of the leftmost column containing each $y$ recovers
$\alpha$: $0,1,2\mapsto0$, $3\mapsto1$, $4,5\mapsto3$, $6,7\mapsto6$.
\begin{figure}[H]
\centering
\begin{tikzpicture}[scale=0.8, every node/.style={font=\small}]
\foreach \i in {0,...,7} {
  \draw (\i,0) rectangle (\i+1,1);
  \node at (\i+0.5,0.5) {$\i$};
}
\foreach \i/\v in {0/1,1/2,2/3,3/4,4/5,6/7} {
  \draw (\i,1) rectangle (\i+1,2);
  \node at (\i+0.5,1.5) {$\v$};
}
\foreach \i/\v in {0/2,1/3,3/5} {
  \draw (\i,2) rectangle (\i+1,3);
  \node at (\i+0.5,2.5) {$\v$};
}
\node[anchor=east] at (-0.3,0.5) {$R_0$};
\node[anchor=east] at (-0.3,1.5) {$R_1$};
\node[anchor=east] at (-0.3,2.5) {$R_2$};
\end{tikzpicture}
\caption{The augmented $N$-tableau of Example~\ref{ex:phi-full-augmented}.}
\label{fig:augmented}
\end{figure}
\end{example}

\subsection{Full \texorpdfstring{$N$}{N}-tableaux represent \texorpdfstring{$\mathrm{Cat}_n$}{Cat\_n}}
\label{sec:full-represent}

\begin{theorem}[$\Phi$ is injective on full $N$-tableaux]
\label{thm:full-injective}
Let $(T_1,\rho)$ and $(T_2,\rho)$ be distinct full $N$-tableaux. Then
$\Phi(\mathrm{col}(T_1,\rho))\ne\Phi(\mathrm{col}(T_2,\rho))$.
\end{theorem}

\begin{proof}
Write $\alpha_1,\alpha_2$ for the two images, and work in the
augmented $N$-tableaux $T_1^+,T_2^+$. These are full $N$-tableaux over
$\{0,\ldots,n\}$ (see the discussion after
Definition~\ref{def:augmented}), and the results above on $N$-tableaux
apply to them verbatim, the alphabet playing no role. A full column
with base $s$ has entries $s,s+1,\ldots$ in rows $0,1,\ldots$, so it is
determined by its height. For $s\in\{0,\ldots,n\}$ write $C^{(1)}_s$
and $C^{(2)}_s$ for the column $C^+_s$ of $T_1^+$ and of $T_2^+$,
respectively. As $T_1\ne T_2$, there is an $s$ with
$C^{(1)}_s\ne C^{(2)}_s$, i.e.\ with different heights. Let $r$ be the
least such $s$; without loss of generality $C^{(1)}_r$ is taller than
$C^{(2)}_r$. Let $i-1$ be the top row of $C^{(2)}_r$ ($i\ge1$, as both
contain row $0$), and let $a:=r+i=C^{(1)}_r(i)$. By the choice of $r$,
$C^{(1)}_s=C^{(2)}_s$ for all $s<r$.

\emph{Case 1: no $C^{(1)}_s$ with $s<r$ contains $a$.} Then
$C^{(1)}_r$ is the leftmost column of $T_1^+$ containing $a$, so
$a\alpha_1=r$ (Corollary~\ref{cor:phi-full-augmented}). No
$C^{(2)}_s=C^{(1)}_s$ with $s<r$ contains $a$ either, and $C^{(2)}_r$
has entries $r,\ldots,a-1$ only. So the leftmost column of $T_2^+$
containing $a$ has base larger than $r$, and $a\alpha_2>r=a\alpha_1$.

\emph{Case 2: some $C^{(1)}_\ell$ with $\ell<r$ contains $a$}, say in
row $t$. Since $C^{(1)}_\ell$ is to the left of $C^{(1)}_r$, which
contains $a$ in row $i$, Corollary~\ref{cor:higher-left} gives $t>i$. As
$C^{(2)}_\ell=C^{(1)}_\ell$, $a$ occurs in row $t$ of $T_2^+$, hence,
by containment, also in row $i$. But row $i-1$ of $T_2^+$ contains
$a-1=C^{(2)}_r(i-1)$, so by Lemma~\ref{lem:successor-above} the entry
$a$ of row $i$ sits directly above it, in $C^{(2)}_r$. This contradicts
$C^{(2)}_r$ ending at row $i-1$, so Case 2 does not occur.
\end{proof}

\begin{corollary}[Full $N$-tableaux represent $\mathrm{Cat}_n$]
\label{cor:full-bijection}
The map $T\mapsto\Phi(\mathrm{col}(T,\rho))$ is a bijection from the
set of full $N$-tableaux over $\mathbf{n}$ onto $\mathrm{Cat}_n$.
Moreover, for every $N$-tableau $T$, its full core $T^f$ is the unique
full $N$-tableau with the same image under $\Phi$.
\end{corollary}

\begin{proof}
Injectivity is Theorem~\ref{thm:full-injective}. For surjectivity, let
$\beta\in\mathrm{Cat}_n$. By Theorem~\ref{thm:volkov-homomorphism},
$\beta=\Phi(T)$ for some $N$-tableau $T$, and $\Phi(T)$ can be computed
from $\mathrm{col}(T,\rho)$ (Corollary~\ref{cor:col-well-defined}). By
Theorem~\ref{thm:phi-full-core},
$\Phi(\mathrm{col}(T,\rho))=\Phi(\mathrm{col}(T^f,\rho))$, and $T^f$ is
a full $N$-tableau (Corollary~\ref{cor:full-core-is-Ntableau} and
Remark~\ref{rem:full-core-tight}). The same argument, applied to any
$N$-tableau $T$, gives the second statement, with uniqueness again by
Theorem~\ref{thm:full-injective}.
\end{proof}

\begin{corollary}[The fibres of $\Phi$]
\label{cor:phi-fibres}
Let $T_1,T_2$ be $N$-tableaux over $\mathbf{n}$. Then
$\Phi(T_1)=\Phi(T_2)$ if and only if $T_1^f=T_2^f$.
\end{corollary}

\begin{proof}
By Corollary~\ref{cor:full-bijection}, $\Phi(T_i)=\Phi(T_i^f)$ for
$i=1,2$. If $T_1^f=T_2^f$, then
$\Phi(T_1)=\Phi(T_1^f)=\Phi(T_2^f)=\Phi(T_2)$. Conversely, if
$\Phi(T_1)=\Phi(T_2)$, then $T_1^f$ and $T_2^f$ are full $N$-tableaux
with the same image under $\Phi$, so $T_1^f=T_2^f$ by
Theorem~\ref{thm:full-injective}.
\end{proof}

\subsection{Image and kernel}
\label{sec:image-kernel}

An element of $\mathrm{Cat}_n$ is conveniently described by two sets.
Let $\alpha\in\mathrm{Cat}_n$. Since $\alpha$ is order-preserving, each
of its kernel classes $\{y : y\alpha=z\}$ is an interval of
$\{0,\ldots,n\}$. Let $X(\alpha)$ be the set of maxima of the kernel
classes, and $Y(\alpha):=\{y\alpha : 0\le y\le n\}$ the image of
$\alpha$. Both have one element per kernel class, so they have the same
size $k$; moreover $n\in X(\alpha)$, and $0\in Y(\alpha)$ since
$0\alpha\le0$. Since the kernel classes are intervals, a letter $x<n$
is the maximum of its class iff $x+1$ lies in a different class, so
\[
X(\alpha)=\{n\}\cup\{x<n : x\alpha\ne(x+1)\alpha\}.
\]
Writing $X(\alpha)=\{x_1<\cdots<x_k\}$,
$Y(\alpha)=\{z_1<\cdots<z_k\}$ and $x_0:=-1$, the kernel classes are
the intervals $\{x_{j-1}+1,\ldots,x_j\}$, and, $\alpha$ being
order-preserving, the $j$-th one is sent to $z_j$. So $\alpha$ is
determined by the pair $(X(\alpha),Y(\alpha))$.

\begin{remark}
\label{rem:image-kernel-literature}
This description goes back to Solomon
\cite[proof of Proposition~3.4]{Solomon}, who uses the same convention
and the same pair of sets. See also
\cite[Proposition~17.20]{SteinbergBook} and
\cite[Section~2]{MargolisSteinberg}, which work with order-increasing
maps; translated to order-decreasing maps by reversing the order, they
record the \emph{minima} of the kernel classes instead of the maxima.
\end{remark}

We now read $X(\alpha)$ and $Y(\alpha)$ off the augmented $N$-tableau.
For a full $N$-tableau $T$, write $c_s$ for the top entry of the
column $C^+_s$ of $T^+$ (Definition~\ref{def:augmented}). Since $T^+$
is full, a column containing the value $a$ in row $i$ has base $a-i$.

\begin{lemma}[Leftmost occurrences of tops]
\label{lem:top-leftmost}
Let $T$ be a full $N$-tableau, and let $y$ be the top of a column
$C^+_{r_1}$ of $T^+$. Then $y$ is also the top of every column
$C^+_{r_2}$ with $r_2<r_1$ that contains $y$. In particular, the
leftmost column containing $y$ has top $y$.
\end{lemma}

\begin{proof}
Let $y$ occur in row $i_1$ of $C^+_{r_1}$ and in row $i_2$ of
$C^+_{r_2}$. By Corollary~\ref{cor:higher-left}, $i_2>i_1$. Suppose $y$ is
not the top of $C^+_{r_2}$; as $C^+_{r_2}$ is full, it has $y+1$ in row
$i_2+1$. By containment, $y+1$ also occurs in row $i_1+1$, and by
Lemma~\ref{lem:successor-above} it sits directly above the entry $y$
of row $i_1$, i.e.\ in $C^+_{r_1}$. This contradicts $y$ being the top
of $C^+_{r_1}$.
\end{proof}

\begin{proposition}[Reading off image and kernel]
\label{prop:image-kernel}
Let $(T,\rho)$ be a full $N$-tableau and $\alpha:=\Phi(\mathrm{col}(T,\rho))$.
Then $X(\alpha)=\{c_0,\ldots,c_n\}$.
\end{proposition}

\begin{proof}
Let $y\in\{0,\ldots,n\}$ and let $C^+_r$ be the leftmost column
containing $y$, so $y\alpha=r$ by
Corollary~\ref{cor:phi-full-augmented}. If $y=n$, then $y\in X(\alpha)$
and $y=c_n$. Let $y<n$.

Suppose $y$ is not a top. Then $C^+_r$, being full, contains $y+1$. No
column $C^+_{r'}$ with $r'<r$ contains $y+1$: its base is $r'<r\le y$,
so, being full, it would contain $y$ as well. Hence
$(y+1)\alpha=r=y\alpha$, and $y\notin X(\alpha)$.

Suppose $y$ is a top. By Lemma~\ref{lem:top-leftmost}, $C^+_r$ has top
$y$, so it does not contain $y+1$. Hence $(y+1)\alpha\ne r=y\alpha$, and
$y\in X(\alpha)$.
\end{proof}

To read off $Y(\alpha)$, let $x\in X(\alpha)$. Then $x$ is a top, so by
Lemma~\ref{lem:top-leftmost} the leftmost column containing $x$ has top
$x$; it is therefore the leftmost column with top $x$, and its base is
$x\alpha$ (Corollary~\ref{cor:phi-full-augmented}). So reading the tops
$c_0,\ldots,c_n$ from left to right and keeping each value only at its
first occurrence gives $X(\alpha)$, and the bases of the columns where
the kept values occur form $Y(\alpha)$.

\begin{example}
In Example~\ref{ex:phi-full-augmented} (Figure~\ref{fig:augmented}),
the tops are $2,3,3,5,5,5,7,7$. Keeping first occurrences gives
$X(\alpha)=\{2,3,5,7\}$, at the columns with bases
$Y(\alpha)=\{0,1,3,6\}$: the kernel classes
$\{0,1,2\},\{3\},\{4,5\},\{6,7\}$ are sent to $0,1,3,6$.
\end{example}

\subsection{The pyramid picture and idempotents}
\label{sec:pyramids}

Let $T$ be a full $N$-tableau. Every entry $a$ in row $i\ge1$ of $T^+$
has $a-1$ in row $i-1$ (directly below it, as its column is full) and
$a$ in row $i-1$ (by containment). So $T^+$ can be redrawn as a wall of
bricks, each brick $a$ of row $i$ resting half on the brick $a-1$ and
half on the brick $a$ of row $i-1$, see
Figure~\ref{fig:pyramids}. Each column of $T^+$ becomes a diagonal
going down to the left, and the picture is a union of pyramids.

A brick $x$ is a \emph{peak} (nothing on top of it) iff it is the top
of its column and the highest, i.e.\ leftmost
(Corollary~\ref{cor:higher-left}), occurrence of $x$. So by
Lemma~\ref{lem:top-leftmost} and Proposition~\ref{prop:image-kernel},
the peaks are the elements of $X(\alpha)$, each occurring once, and
sliding down to the left from a peak $x$ to row $0$ reaches
$x\alpha\in Y(\alpha)$.

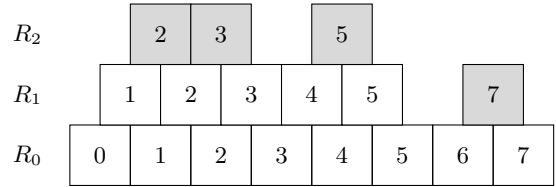
\begin{figure}[H]
\centering
\begin{tikzpicture}[scale=0.8, every node/.style={font=\small}]
\foreach \v in {0,...,7} {
  \draw (\v-0.5,0) rectangle (\v+0.5,1);
  \node at (\v,0.5) {$\v$};
}
\foreach \v/\c in {1/white,2/white,3/white,4/white,5/white,7/gray!30} {
  \draw[fill=\c] (\v-1,1) rectangle (\v,2);
  \node at (\v-0.5,1.5) {$\v$};
}
\foreach \v in {2,3,5} {
  \draw[fill=gray!30] (\v-1.5,2) rectangle (\v-0.5,3);
  \node at (\v-1,2.5) {$\v$};
}
\node[anchor=east] at (-0.8,0.5) {$R_0$};
\node[anchor=east] at (-0.8,1.5) {$R_1$};
\node[anchor=east] at (-0.8,2.5) {$R_2$};
\end{tikzpicture}
\caption{The pyramid picture of the augmented $N$-tableau of
Figure~\ref{fig:augmented}. The peaks (shaded) are $2,3,5,7$, giving
$X(\alpha)=\{2,3,5,7\}$; sliding down to the left from them reaches
$0,1,3,6$, giving $Y(\alpha)$.}
\label{fig:pyramids}
\end{figure}

The pyramid picture also detects idempotents. Since each brick of row
$i\ge1$ rests on two bricks of row $i-1$, a peak $x$ in row $i$ has a
full pyramid beneath it, whose base consists of the bricks
$x-i,\ldots,x$ of row $0$. As noted above, $x-i=x\alpha$,
so the base is $\{x\alpha,\ldots,x\}$. Note that the base lies in the
added row $0$ of $T^+$, not in row $1$ of $T$.

\begin{proposition}[Idempotents]
\label{prop:idempotents}
Let $(T,\rho)$ be a full $N$-tableau and $\alpha:=\Phi(\mathrm{col}(T,\rho))$.
Then $\alpha$ is idempotent if and only if, in the pyramid picture of
the augmented tableau $T^+$, the bases of the pyramids of the peaks,
taken in the added row $0$, are pairwise disjoint.
\end{proposition}

\begin{proof}
Write $X(\alpha)=\{x_1<\cdots<x_k\}$ and $x_0:=-1$. The map $\alpha$ is
idempotent iff it fixes every point $x_j\alpha$ of its image, i.e.\ iff
$x_j\alpha$ lies in the kernel class $\{x_{j-1}+1,\ldots,x_j\}$ of
$x_j$; as $x_j\alpha\le x_j$, this means $x_j\alpha>x_{j-1}$. This is
exactly the condition that the base $\{x_j\alpha,\ldots,x_j\}$ does not
meet the bases of the earlier peaks, which lie in $\{0,\ldots,x_{j-1}\}$
and include $x_{j-1}$.
\end{proof}

\begin{example}
In Figure~\ref{fig:pyramids}, the peak $3$ is in row $2$, so its base
is $\{1,2,3\}$, while its kernel class is $\{3\}$; the base overlaps
that of the peak $2$. Accordingly, $\alpha$ is not idempotent:
$1=3\alpha$ lies in the image but $1\alpha=0$. By contrast, let $n=7$
and let $T=(R_1,R_2)$ be the full $N$-tableau with $R_1=(1,2,4,5,7)$ and
$R_2=(2,5)$. Its pyramid
picture (Figure~\ref{fig:pyramids-idempotent}) has peaks $2,5,7$ with
disjoint bases $\{0,1,2\},\{3,4,5\},\{6,7\}$, so $\alpha$ is
idempotent; indeed, it sends these three sets to $0,3,6$.
\end{example}

\begin{figure}[H]
\centering
\begin{tikzpicture}[scale=0.8, every node/.style={font=\small}]
\foreach \v in {0,...,7} {
  \draw (\v-0.5,0) rectangle (\v+0.5,1);
  \node at (\v,0.5) {$\v$};
}
\foreach \v/\c in {1/white,2/white,4/white,5/white,7/gray!30} {
  \draw[fill=\c] (\v-1,1) rectangle (\v,2);
  \node at (\v-0.5,1.5) {$\v$};
}
\foreach \v in {2,5} {
  \draw[fill=gray!30] (\v-1.5,2) rectangle (\v-0.5,3);
  \node at (\v-1,2.5) {$\v$};
}
\node[anchor=east] at (-0.8,0.5) {$R_0$};
\node[anchor=east] at (-0.8,1.5) {$R_1$};
\node[anchor=east] at (-0.8,2.5) {$R_2$};
\end{tikzpicture}
\caption{The pyramid picture of an idempotent: the peaks $2,5,7$
stand on the disjoint bases $\{0,1,2\}$, $\{3,4,5\}$, $\{6,7\}$.}
\label{fig:pyramids-idempotent}
\end{figure}

\subsection{A complete example}
\label{sec:complete-example}

We follow one $N$-tableau through the whole process. Let $n=9$ and
\[
R_1=(1,2,3,4,5,6,8),\quad R_2=(2,3,5,6,8),\quad R_3=(3,5,8),\quad
R_4=(5,8).
\]
This is an $N$-tableau: $R_4\subseteq R_3\subseteq R_2\subseteq R_1$
and $1<2<3<5$. Figure~\ref{fig:complete-trivial} draws it in its
trivial positioning, i.e.\ as a classical tableau.

\begin{figure}[H]
\centering
\begin{tikzpicture}[scale=0.7, every node/.style={font=\small}]
\foreach \i/\v in {1/1,2/2,3/3,4/4,5/5,6/6,7/8} {
  \draw (\i-1,0) rectangle (\i,1); \node at (\i-0.5,0.5) {$\v$};
}
\foreach \i/\v in {1/2,2/3,3/5,4/6,5/8} {
  \draw (\i-1,1) rectangle (\i,2); \node at (\i-0.5,1.5) {$\v$};
}
\foreach \i/\v in {1/3,2/5,3/8} {
  \draw (\i-1,2) rectangle (\i,3); \node at (\i-0.5,2.5) {$\v$};
}
\foreach \i/\v in {1/5,2/8} {
  \draw (\i-1,3) rectangle (\i,4); \node at (\i-0.5,3.5) {$\v$};
}
\foreach \r in {1,...,4} {\node[anchor=east] at (-0.3,\r-0.5) {$R_\r$};}
\end{tikzpicture}
\caption{The $N$-tableau of Section~\ref{sec:complete-example}, in its
trivial positioning.}
\label{fig:complete-trivial}
\end{figure}

\emph{Step 1: tight positioning.} Each entry moves to sit directly
above the largest smaller entry of the row below. The resulting columns,
from bottom to top, are
\[
C_1=(1,2,3,5),\ C_2=(2,3,5,8),\ C_3=(3),\ C_4=(4,5),\
C_5=(5,6,8),\ C_6=(6,8),\ C_7=(8);
\]
see Figure~\ref{fig:complete-tight}, left.

\emph{Step 2: full core.} The columns $C_3$, $C_4$ and $C_7$ are
full, and the others have gaps: $C_1$ at row $4$, $C_2$ at row $3$
(and again at row $4$), $C_5$ at row $3$, and $C_6$ at row $2$. So
$g_1=3$, $g_2=2$, $g_3=1$, $g_4=2$, $g_5=2$, $g_6=1$, $g_7=1$.
Deleting the five cells above these rows (shaded in
Figure~\ref{fig:complete-tight}) leaves the full $N$-tableau
\[
R^f_1=(1,2,3,4,5,6,8),\qquad R^f_2=(2,3,5,6),\qquad R^f_3=(3),
\]
and by Theorem~\ref{thm:phi-full-core} it has the same image under
$\Phi$.

\begin{figure}[H]
\centering
\begin{minipage}{0.45\textwidth}
\centering
\begin{tikzpicture}[scale=0.7, every node/.style={font=\small}]
\foreach \i/\v in {1/1,2/2,3/3,4/4,5/5,6/6,7/8} {
  \draw (\i-1,0) rectangle (\i,1); \node at (\i-0.5,0.5) {$\v$};
}
\foreach \i/\v/\c in {1/2/white,2/3/white,4/5/white,5/6/white,6/8/gray!30} {
  \draw[fill=\c] (\i-1,1) rectangle (\i,2); \node at (\i-0.5,1.5) {$\v$};
}
\foreach \i/\v/\c in {1/3/white,2/5/gray!30,5/8/gray!30} {
  \draw[fill=\c] (\i-1,2) rectangle (\i,3); \node at (\i-0.5,2.5) {$\v$};
}
\foreach \i/\v in {1/5,2/8} {
  \draw[fill=gray!30] (\i-1,3) rectangle (\i,4); \node at (\i-0.5,3.5) {$\v$};
}
\foreach \r in {1,...,4} {\node[anchor=east] at (-0.3,\r-0.5) {$R_\r$};}
\end{tikzpicture}

\smallskip
Tight positioning; deleted cells shaded.
\end{minipage}
\hfill
\begin{minipage}{0.45\textwidth}
\centering
\begin{tikzpicture}[scale=0.7, every node/.style={font=\small}]
\foreach \i/\v in {1/1,2/2,3/3,4/4,5/5,6/6,7/8} {
  \draw (\i-1,0) rectangle (\i,1); \node at (\i-0.5,0.5) {$\v$};
}
\foreach \i/\v in {1/2,2/3,4/5,5/6} {
  \draw (\i-1,1) rectangle (\i,2); \node at (\i-0.5,1.5) {$\v$};
}
\foreach \i/\v in {1/3} {
  \draw (\i-1,2) rectangle (\i,3); \node at (\i-0.5,2.5) {$\v$};
}
\foreach \r in {1,...,3} {\node[anchor=east] at (-0.3,\r-0.5) {$R^f_\r$};}
\end{tikzpicture}

\smallskip
The full core.
\end{minipage}
\caption{Steps 1 and 2 for the $N$-tableau of
Figure~\ref{fig:complete-trivial}.}
\label{fig:complete-tight}
\end{figure}

\emph{Step 3: image and kernel.} In the augmented $N$-tableau of the
full core (Figure~\ref{fig:complete-augmented}, left), the tops of the
columns $C^+_0,\ldots,C^+_9$ are
\[
3,3,3,5,6,6,6,8,8,9.
\]
Keeping first occurrences (Proposition~\ref{prop:image-kernel}) gives
$X(\alpha)=\{3,5,6,8,9\}$, at the columns with bases
$Y(\alpha)=\{0,3,4,7,9\}$. So the kernel classes
$\{0,1,2,3\},\{4,5\},\{6\},\{7,8\},\{9\}$ are sent to $0,3,4,7,9$
respectively:
\[
\alpha=\begin{pmatrix}
0&1&2&3&4&5&6&7&8&9\\
0&0&0&0&3&3&4&7&7&9
\end{pmatrix}.
\]
The fixed points of $\alpha$ are $0,7,9$, exactly the letters that
do not occur in $T$ (Corollary~\ref{cor:phi-full-height}).

\emph{Step 4: the pyramid picture}
(Figure~\ref{fig:complete-augmented}, right). The peaks are $3$ (in
row $3$), $5$ and $6$ (row $2$), $8$ (row $1$) and $9$ (row $0$),
giving $X(\alpha)$; sliding down to the left from them reaches
$3-3=0$, $5-2=3$, $6-2=4$, $8-1=7$ and $9-0=9$, giving $Y(\alpha)$
(Section~\ref{sec:pyramids}). The bases $\{0,1,2,3\}$ and $\{3,4,5\}$
of the peaks $3$ and $5$ meet, so $\alpha$ is not idempotent
(Proposition~\ref{prop:idempotents}).

\begin{figure}[H]
\centering
\begin{minipage}{0.48\textwidth}
\centering
\begin{tikzpicture}[scale=0.55, every node/.style={font=\small}]
\foreach \v in {0,...,9} {
  \draw (\v,0) rectangle (\v+1,1); \node at (\v+0.5,0.5) {$\v$};
}
\foreach \v in {1,2,3,4,5,6,8} {
  \draw (\v-1,1) rectangle (\v,2); \node at (\v-0.5,1.5) {$\v$};
}
\foreach \v in {2,3,5,6} {
  \draw (\v-2,2) rectangle (\v-1,3); \node at (\v-1.5,2.5) {$\v$};
}
\draw (0,3) rectangle (1,4); \node at (0.5,3.5) {$3$};
\foreach \r in {0,...,3} {\node[anchor=east] at (-0.2,\r+0.5) {$R_\r$};}
\end{tikzpicture}

\smallskip
The augmented full core.
\end{minipage}
\hfill
\begin{minipage}{0.48\textwidth}
\centering
\begin{tikzpicture}[scale=0.55, every node/.style={font=\small}]
\foreach \v/\c in {0/white,1/white,2/white,3/white,4/white,5/white,6/white,7/white,8/white,9/gray!30} {
  \draw[fill=\c] (\v-0.5,0) rectangle (\v+0.5,1); \node at (\v,0.5) {$\v$};
}
\foreach \v/\c in {1/white,2/white,3/white,4/white,5/white,6/white,8/gray!30} {
  \draw[fill=\c] (\v-1,1) rectangle (\v,2); \node at (\v-0.5,1.5) {$\v$};
}
\foreach \v/\c in {2/white,3/white,5/gray!30,6/gray!30} {
  \draw[fill=\c] (\v-1.5,2) rectangle (\v-0.5,3); \node at (\v-1,2.5) {$\v$};
}
\draw[fill=gray!30] (1,3) rectangle (2,4); \node at (1.5,3.5) {$3$};
\end{tikzpicture}

\smallskip
The pyramid picture; peaks shaded.
\end{minipage}
\caption{Steps 3 and 4: reading off $X(\alpha)=\{3,5,6,8,9\}$ and
$Y(\alpha)=\{0,3,4,7,9\}$.}
\label{fig:complete-augmented}
\end{figure}
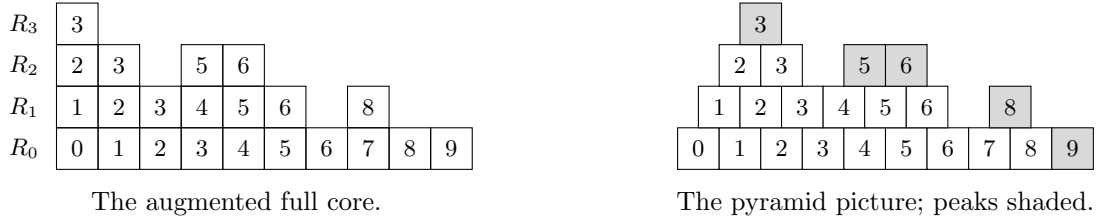

\section*{AI disclosure}

The mathematical content of this paper -- the ideas, the definitions,
the results, and the construction of the proofs -- is the work of the
author. The AI models Claude Sonnet and Claude Opus (Anthropic), used
through Claude Code, wrote computer code that was used to test
conjectures empirically before they were proved, and wrote the text
and the figures of the paper. The author has verified all of the content and is
responsible for its correctness.

\bibliographystyle{plain}
\bibliography{StylicToCatalan}

\begin{thebibliography}{10}

\bibitem{AR}
A.~Abram and C.~Reutenauer.
\newblock The stylic monoid.
\newblock {\em Semigroup Forum}, 105(1):1--45, 2022.

\bibitem{AHMNT}
Antoine Abram, Florent Hivert, James~D. Mitchell, Jean-Christophe Novelli, and
  Maria Tsalakou.
\newblock Power quotients of plactic-like monoids.
\newblock In {\em Proceedings of the 13th edition of the conference on {R}andom
  {G}eneration of {C}ombinatorial {S}tructures. {P}olyominoes and {T}ilings},
  volume 403 of {\em Electron. Proc. Theor. Comput. Sci. (EPTCS)}, pages
  12--17. EPTCS, 2024.

\bibitem{ARS}
Antoine Abram, Christophe Reutenauer, and Franco~V. Saliola.
\newblock Quivers of stylic algebras.
\newblock {\em Algebr. Comb.}, 6(6):1621--1635, 2023.

\bibitem{AirdRibeiro}
Thomas Aird and Duarte Ribeiro.
\newblock Tropical representations and identities of the stylic monoid.
\newblock {\em Semigroup Forum}, 106(1):1--23, 2023.

\bibitem{Choffrut}
Christian Choffrut.
\newblock Grammic monoids with three generators.
\newblock {\em Semigroup Forum}, 105(3):680--692, 2022.

\bibitem{GM}
Olexandr Ganyushkin and Volodymyr Mazorchuk.
\newblock {\em Classical finite transformation semigroups}, volume~9 of {\em
  Algebra and Applications}.
\newblock Springer-Verlag London, Ltd., London, 2009.
\newblock An introduction.

\bibitem{Hage}
Nohra Hage.
\newblock Coherent presentations for stylic monoids via row rewriting, 2026.
\newblock arXiv:2609.15515.

\bibitem{Higgins93}
Peter~M. Higgins.
\newblock Combinatorial results for semigroups of order-preserving mappings.
\newblock {\em Math. Proc. Cambridge Philos. Soc.}, 113(2):281--296, 1993.

\bibitem{HivertThiery}
Florent Hivert and Nicolas~M. Thi\'ery.
\newblock The {H}ecke group algebra of a {C}oxeter group and its representation
  theory.
\newblock {\em J. Algebra}, 321(8):2230--2258, 2009.

\bibitem{Lothaire}
M.~Lothaire.
\newblock {\em Algebraic combinatorics on words}, volume~90 of {\em
  Encyclopedia of Mathematics and its Applications}.
\newblock Cambridge University Press, Cambridge, 2002.

\bibitem{MargolisSteinberg}
Stuart Margolis and Benjamin Steinberg.
\newblock The algebra of the {C}atalan monoid as an incidence algebra: a simple
  proof, 2018.
\newblock arXiv:1806.06531.

\bibitem{Solomon}
Andrew Solomon.
\newblock {C}atalan monoids, monoids of local endomorphisms, and their
  presentations.
\newblock {\em Semigroup Forum}, 53(3):351--368, 1996.

\bibitem{SteinbergBook}
Benjamin Steinberg.
\newblock {\em Representation theory of finite monoids}.
\newblock Universitext. Springer, Cham, 2016.

\bibitem{Volkov}
Mikhail~V. Volkov.
\newblock Identities of the stylic monoid.
\newblock {\em Semigroup Forum}, 105(1):345--349, 2022.

\end{thebibliography}

\end{document}